\documentclass[3p]{elsarticle}

\usepackage{bm}
\usepackage[parfill]{parskip}
\usepackage[utf8]{inputenc}
\usepackage{url}
\usepackage{amsmath,amssymb,amsfonts,amsthm}
\usepackage{color}
\usepackage{bm}
\usepackage{graphicx}
\usepackage{subcaption}
\usepackage[linesnumbered,ruled]{algorithm2e}

\usepackage{lipsum}
\usepackage{amsfonts}
\usepackage{graphicx}
\usepackage{epstopdf}
\usepackage{algorithmic}
\ifpdf
  \DeclareGraphicsExtensions{.eps,.pdf,.png,.jpg}
\else
  \DeclareGraphicsExtensions{.eps}
\fi

\newtheorem{theorem}{Theorem}

\DeclareMathOperator*{\argmin}{arg\,min}

\makeatletter
\newsavebox\myboxA
\newsavebox\myboxB
\newlength\mylenA

\newcommand*\xoverline[2][0.75]{%
    \sbox{\myboxA}{$\m@th#2$}%
    \setbox\myboxB\null
    \ht\myboxB=\ht\myboxA%
    \dp\myboxB=\dp\myboxA%
    \wd\myboxB=#1\wd\myboxA
    \sbox\myboxB{$\m@th\overline{\copy\myboxB}$}
    \setlength\mylenA{\the\wd\myboxA}
    \addtolength\mylenA{-\the\wd\myboxB}%
    \ifdim\wd\myboxB<\wd\myboxA%
       \rlap{\hskip 0.5\mylenA\usebox\myboxB}{\usebox\myboxA}%
    \else
        \hskip -0.5\mylenA\rlap{\usebox\myboxA}{\hskip 0.5\mylenA\usebox\myboxB}%
    \fi}
\makeatother

\journal{arXiv.org}

\newcommand{\TheTitle}{Filtered Stochastic Galerkin Methods For Hyperbolic Equations} 

\date{\today}

\usepackage{amsopn}

\begin{document}

\begin{frontmatter}

\title{\TheTitle}

\author[adressJonas]{Jonas Kusch}
\author[adressRyan]{Ryan G. McClarren}
\author[adressMartin]{Martin Frank}

\address[adressJonas]{Karlsruhe Institute of Technology, Karlsruhe,
    jonas.kusch@kit.edu}
\address[adressRyan]{University of Notre Dame, Notre Dame,  rmcclarr@nd.edu}
\address[adressMartin]{Karlsruhe Institute of Technology, Karlsruhe,
    martin.frank@kit.edu}

\begin{abstract}
Uncertainty Quantification for nonlinear hyperbolic problems becomes a challenging task in the vicinity of shocks. Standard intrusive methods lead to oscillatory solutions and can result in non-hyperbolic moment systems. The intrusive polynomial moment (IPM) method guarantees hyperbolicity but comes at higher numerical costs. In this paper, we filter the gPC coefficients of the Stochastic Galerkin (SG) approximation, which allows a numerically cheap reduction of oscillations. The derived filter is based on Lasso regression which sets small gPC coefficients of high order to zero. We adaptively choose the filter strength to obtain a zero-valued highest order moment, which allows optimality of the corresponding optimization problem. The filtered SG method is tested for Burgers' and the Euler equations. Results show a reduction of oscillations at shocks, which leads to an improved approximation of expectation values and the variance compared to SG and IPM.
\end{abstract}


\end{frontmatter}


\section{Introduction}
Systems of hyperbolic equations play a key role in various research areas, including Euler equations in fluid dynamics \cite{toro2013riemann}, magnetohydrodynamics (MHD) equations in plasma physics \cite{sutton2006engineering}, and radiation-hydrodynamics in astrophysics \cite{lowrie1999coupling,mcclarren2008manufactured}.
Though numerical methods for these problems are a continuing area for research, efficient and accurate methods exist for many problems.

Given that accurate numerical solutions exist, computational scientists are increasingly concerned with how uncertainties in the ``input'' data, such as initial/boundary conditions, constitutive relations, and other parameters affect the conclusions drawn from computer simulations. Answering such questions are the purview of the field of uncertainty quantification (UQ) \cite{mcclarrenUQ}.


A general stochastic hyperbolic equation takes the form
\begin{subequations}\label{eq:hyperbolicProblem}
\begin{align}
\partial_t \bm{u}(t,\bm{x},\xi) + \nabla&\cdot\bm{F}(\bm{u}(t,\bm{x},\xi)) = \bm{0}, \\ \label{eq:ic}
\bm{u}(t=0,\bm{x},&\xi) = \bm{u}_{\text{IC}}(\bm{x},\xi).
\end{align}
\end{subequations}
The solution $\bm{u}\in\mathbb{R}^p$ depends on time $t\in\mathbb{R}_+$, physical space $\bm{x}\in\mathcal{D}\subset \mathbb{R}^d$ and the random variable $\xi\in\Theta$, where $\xi$ is distributed according to the probability density function $f_{\Xi}:\Theta\to\mathbb{R}$. In this work, we assume a one-dimensional random variable, however an extension to higher order is possible. The physical flux is $\bm{F} = (\bm{f}_1,\cdots,\bm{f}_d)$ with $\bm{f}_i\in\mathbb{R}^{p}$. Furthermore, initial- and boundary conditions are needed: the initial conditions are specified in \eqref{eq:ic}, and we will specify boundary conditions as needed for particular systems. 

Methods to quantify the uncertainty in the solution due to the random variable $\xi$  can be divided into intrusive and non-intrusive methods. Non-intrusive methods involve using an existing solution technique to compute solutions that are then used in a post-processing fashion to estimate uncertainties.

In this paper, we focus on intrusive methods where the random variables are included in the solution of the model equations. The most commonly used Stochastic Galerkin (SG) \cite{ghanem2003stochastic} method uses the general polynomial chaos (gPC) expansion \cite{wiener1938homogeneous} to represent the solution: Using polynomials $\bm{\varphi} = (\varphi_0,\cdots,\varphi_N)^T$ which are orthonormal w.r.t.\ the probability density function $f_{\Xi}$, the solution can be approximated by
\begin{align}\label{eq:gPC}
\bm{u} \approx \bm{u}_{N} = \sum_{i=0}^N \bm{\hat{u}}_i(t,\bm{x})\varphi_i(\xi).
\end{align}
The gPC expansion coefficients $\bm{\hat{u}}_i = (\hat{u}_{si})_{s = 1,\cdots,p}$ are often called moments and can be used to compute the expected value as well as the variance of the solution by
\begin{align*}
\text{E}[\bm{u}_{N}] &= (\hat{u}_{s0})_{s = 1,\cdots,p}, \\ \text{Var}[\bm{u}_{N}] = \text{E}[\bm{u}_{N}^2] &- \text{E}[\bm{u}_{N}]^2 = \left(\sum_{i = 1}^N \hat{u}_{si}^2\right)_{s = 1,\cdots,p}.
\end{align*}
To derive a time evolution equation for the moments, the SG method plugs \eqref{eq:gPC} into \eqref{eq:hyperbolicProblem} and projects the resulting residual to zero. The SG moment system then reads
\begin{subequations}\label{eq:SGMomentSystem}
\begin{align}
&\partial_t \bm{\hat{u}} + \int_{\Theta} \nabla\cdot\bm{F}(\bm{u}_N)\bm{\varphi}f_{\Xi}\,d\xi = \bm{0}, \\
&\bm{\hat{u}}(t=0,\bm{x}) = \int_{\Theta}\bm{u}_{\text{IC}}(\bm{x},\xi)f_{\Xi}\,d\xi.
\end{align}
\end{subequations}
SG promises pseudo-spectral convergence for smooth data \cite{canuto1982approximation}, which can be seen in simple applications such as \cite{xiu2002modeling,gottlieb2008galerkin}. Spectral convergence has also been proven in \cite{jin2017uniform} for kinetic equations and in \cite{despres2013robust} for Burgers' equation assuming sufficiently smooth solutions.
However, solutions to hyperbolic equations usually exhibit discontinuities, leading to slow convergence rates as well as oscillations from Gibbs phenomenon \cite{poette2009uncertainty}. Moreover, the expectation value obtained with SG shows an incorrect discontinuous profile. Adding numerical diffusion yields a continuous expectation value, but results in a poorly resolved numerical solution \cite{pettersson2009numerical}. Especially regions with deterministic shocks are poorly approximated if too much diffusion is added by the numerical discretization. Furthermore, the SG moment system may not be hyperbolic, which makes the use of standard methods impossible \cite{despres2013robust}.
To ensure hyperbolicity of the moment system, the intrusive polynomial moment (IPM) method has been introduced in \cite{poette2009uncertainty}. Applications of the IPM method can be found in \cite{poette2009uncertainty,poette2011treatment,despres2013robust,kusch2017maximum,schlachter2017hyperbolicity,kusch2018intrusive}. The core idea of IPM is to not expand the conserved variables $\bm{u}$, but the entropy variables. Defining the entropy variables $\bm{v} = U'(\bm{u})$, where $U$ is a strictly convex entropy of \eqref{eq:hyperbolicProblem} lets us write the conserved variables $\bm{u}$ in terms of entropy variables, i.e. $\bm{u} = \bm{u}(\bm{v})$. Rewriting \eqref{eq:hyperbolicProblem} in terms of entropy variables yields
\begin{align*}
\partial_t \bm{u}(\bm{v}) + \nabla\cdot\bm{F}(\bm{u}(\bm{v})) = \bm{0}.
\end{align*}
Expanding $\bm{v}$ with gPC polynomials, i.e.
\begin{align*}
\bm{v} \approx \bm{v}_{N} = \sum_{i=0}^N \bm{\hat{v}}_i(t,\bm{x})\varphi_i(\xi)
\end{align*}
and again projecting the residual to zero gives the IPM moment system
\begin{align*}
\partial_t \bm{\hat{u}}(\bm{\hat{v}}) + \int_{\Theta} \nabla\cdot\bm{F}(\bm{u}(\bm{v}_N))\bm{\varphi}f_{\Xi}\,d\xi = \bm{0}, \text{ with }\enskip\bm{\hat{u}}(\bm{\hat{v}}(t=0,\bm{x})) = \int_{\Theta}\bm{u}_{\text{IC}}(\bm{x},\xi)f_{\Xi}\,d\xi.
\end{align*}
To obtain the gPC coefficients of the entropy variables from the moments, one needs to solve the dual problem
\begin{align}\label{eq:dualProblem}
 \bm{\hat v}(\bm{\hat u}) := \argmin_{\bm{\lambda} \in \mathbb{R}^{p\times N+1}}
  \langle U_*(\bm{\lambda} \bm{\varphi})\rangle - \sum_{s,i}\lambda_{si} \hat u_{si},
\end{align}
where $U_*$ is the Legendre transformation of the entropy $U$. 
The IPM system is hyperbolic \cite{poette2009uncertainty} and guarantees a maximum principle for scalar problems, which prohibits oscillatory over- and undershoots at the maximal and minimal solution values \cite{kusch2017maximum}. However, since the dual problem needs to be solved repeatedly, the IPM is numerically costly compared to the SG method. Furthermore, in the system case, the IPM solution still suffers from oscillations as we will show in Section \ref{sec:results}.

In this paper, we propose to filter the coefficients of the gPC expansion \eqref{eq:gPC} in order to dampen oscillations, similar to \cite{kusch2018intrusive}. Filters are a common strategy to reduce oscillations in spectral methods, see for example \cite{boyd2001chebyshev,hesthaven2007spectral}. Applications of filters in the context of kinetic theory can be found in \cite{mcclarren2010robust}, where a filter is constructed by adding a penalizing term to the L$^2$ error of the solution. 
More choices of filter functions can be found in \cite{frank2016convergence,laboure2016implicit,radice2013new}, however the task of constructing an adequate filter strength for a chosen discretization remains a cumbersome task that is typically user-determined and problem dependent. 

In this work, we apply filters to SG in order to mitigate oscillations in regions with uncertainty, while maintaining high order accuracy in deterministic regions. Additionally, we construct a new filter, which is based on Lasso regression \cite{tibshirani1996regression}. The resulting filter depends on the gPC coefficients and sets small and high order coefficients to zero, yielding sparsity in the filtered coefficients. We use this property to adaptively pick the filter strength such that the moment with highest order is set to zero. This automated choice of the filter strength avoids the tedious task of picking a suitable filter parameter and at the same time promises optimality of the optimization problem, i.e. we obtain a minimal value of the penalized L$^2$ error. We demonstrate the effectiveness of our method by investigating Burgers' and the Euler equations and comparing the results to SG and IPM. One observes that the filtered method outperforms Stochastic Galerkin in the Burgers' case. Due to its reduced runtime, the filtered SG is able to compete with IPM. When taking a look at the Euler equations, the filtered SG yields an improved approximation of shocks.

The paper is structured as follows: In Section \ref{sec:Filters}, we review the concept of filtering and present the new Lasso filter. Section \ref{sec:numerics} discusses the numerical implementation of the filter. An automated way to adaptively pick the filter strength is derived in Section \ref{sec:filterStrength}. The filtered solution is compared to common SG and IPM by investigating Burgers' as well as the Euler equations in Section \ref{sec:results}. In Section \ref{sec:ConclusionOutlook}, we sum up our findings and give an outlook on future work.

\section{Filters for Uncertainty Quantification}
\label{sec:Filters}
In the following, we apply the concept of filtering to UQ and derive the standard $L^2$ filter. The main idea of filters is to dampen high order expansion coefficients in the gPC expansion \eqref{eq:gPC}. Optimality with respect to the L$^2$ norm when approximating a function $\bm{u}$ is achieved if the expansion coefficients are chosen to minimize the cost function 
\begin{align*}
\mathcal{J} := \frac{1}{2}\int_{\Theta} \Vert\bm{u}-\bm{u}_N\Vert_2^2 f_{\Xi} \,d\xi,
\end{align*}
where $\Vert\cdot\Vert_2$ is the Euclidean norm. The minimizer is then given by $\bm{\hat{u}}_i = \int_{\Theta} \bm{u} \varphi_i f_{\Xi} \,d\xi$. This choice suffers from oscillations when the function $\bm{u}$ lacks sufficient regularity. The filtered gPC expansion tackles this problem. It is given by
\begin{align}\label{eq:FilteredGPC}
\bm{u} \approx \bm{u}_{N}^F := \sum_{i=0}^N g(i)\bm{\hat{u}}_i(t,\bm{x})\varphi_i(\xi),
\end{align}
where $g$ is the filter function. Defining the filtered gPC coefficients to be $\hat u_{si}^F := g(i)\hat u_{si}$, a filter function can be constructed by minimizing
\begin{align}\label{eq:costFunction}
\mathcal{J}_{\lambda} := \frac{1}{2}\int_{\Theta} \Vert\bm{u}-\bm{u}_N^F\Vert_2^2 f_{\Xi} \,d\xi + \lambda \int_{\Theta} \left\Vert L\bm{u}_N^F\right\Vert^2_2 f_{\Xi} \,d\xi,  \qquad \lambda \geq 0,
\end{align}
over the filtered coefficients. The operator $L$ is commonly chosen to punish oscillations. A standard choice for uniform distributions is $Lu(\xi) = ((1-\xi^2)u'(\xi))'$, since 
\begin{align*}
L\varphi_i = -i(i+1)\varphi_i,
\end{align*}
i.e.\ the Legendre polynomials\footnote{For arbitrary distributions, the operator should be chosen s.t. the corresponding gPC polynomials are eigenfunctions of $L$.} are eigenfunctions of $L$. For ease of exposition we assume a scalar problem, i.e.\ $p=1$. Extending the results to systems is straight forward. The filtered solution representation is now
\begin{align*}
u \approx u_{N}^F := \sum_{i=0}^N \hat{u}_i^F \varphi_i = \sum_{i=0}^N g(i) \hat{u}_i \varphi_i,
\end{align*}
where the filter function $g(i)$ damps high order expansion coefficients. Differentiation of \eqref{eq:costFunction} w.r.t. $\hat{u}_i^F$ gives
\begin{align*}
\hat{u}_i^F = \frac{1}{1+\lambda i^2(i+1)^2}\hat{u}_i,
\end{align*}
i.e. the filter function is
\begin{align}\label{eq:L2FilterFunction}
g(i) = \frac{1}{1+\lambda i^2(i+1)^2}.
\end{align}
This corresponds to the $L^2$ filter based on splines\footnote{This is not the only filter which damps oscillations. Indeed, several other filters can be used such as Lanczos and ErfcLog\cite{radice2013new}.}. One can see that the filter damps high order coefficients, while leaving the $0_{th}$ order coefficient untouched, meaning that we maintain the conservation property. 

The filter strength $\lambda$ must be chosen such that oscillations are sufficiently dampened while the solution structure is preserved. Finding an adequate filter strength is challenging. In the following, we derive a filter which can be used to choose this filter strength.
\subsection{Construction of the Lasso Filter}
Our task is to find a smooth representation of $\bm{u}$ which promotes sparsity. Combining the ideas of Lasso regression and filtering, we introduce the cost functional
\begin{align}\label{eq:LassoFunctional}
\mathcal{J}_{\lambda}(\bm{\hat{u}}^F) = \frac{1}{2}\int_{\Theta} \left( u-\sum_{i=0}^N \hat{u}_i^F \varphi_i \right)^2 f_{\Xi}\,d\xi + \lambda \int_{\Theta}\sum_{i=0}^N \left\vert L\hat{u}_i^F \varphi_i \right\vert f_{\Xi}\,d\xi.
\end{align}
Compared to L$^2$ filtering, the punishing term has been changed to an L$^1$ term which acts on each expansion term of the solution individually. The corresponding filter takes the following form:
\begin{theorem}
The filter that corresponds to the cost functional \eqref{eq:LassoFunctional} is given by 
\begin{align}\label{eq:LassoFilterFunction}
g(i,\hat u_i) = \left(1 - \frac{\lambda i(i+1)\Vert \varphi_i \Vert_{L^1}}{\vert \hat u_i \vert}\right)_+,
\end{align}
where $(\cdot)_+$ is defined as
\begin{align*}
x_+ = 
\begin{cases}
x & x>0 \\
0 & \mathrm{ otherwise}
\end{cases}.
\end{align*}
\end{theorem}
\begin{proof}
The proof follows standard ideas from Lasso regression, see \cite{tibshirani1996regression}. To minimize the cost functional \eqref{eq:LassoFunctional}, we compute the subdifferential\footnote{At differentiable points, the subdifferential of $f$ is simply the gradient. At non-differentiable points $x_0$, the subdifferential is the set of slopes belonging to all straight lines that touch $f(x_0)$ and lie below $f$ in the neighborhood of $f(x_0)$.}
\begin{align}\label{eq:LassoFilterFunctionGrad}
\partial_{j} \mathcal{J}_{\lambda}(\bm{\hat{u}}^F) =
\begin{cases}
\left\{ -\int_{\Theta} \left( u-\sum_{i} \hat u_i^F \varphi_i \right) \varphi_j f_{\Xi}\,d\xi + \lambda\eta \int_{\Theta} \left\vert L \varphi_j \right\vert f_{\Xi}\,d\xi : \eta \in [-1,1] \right\} & \enskip \text{ if } \hat u_j^F = 0 \\
-\int_{\Theta} \left( u-\sum_{i} \hat u_i^F \varphi_i \right) \varphi_j f_{\Xi}\,d\xi + \lambda\cdot\text{sign}(u_i^F) \int_{\Theta} \left\vert L \varphi_i \right\vert f_{\Xi}\,d\xi & \enskip \text{ else }
\end{cases}.
\end{align}
We now need to pick $\hat u_j^F$, s.t.\  the cost function \eqref{eq:LassoFunctional} is minimized. Assuming $\hat u_j^F = 0$, this translates into ensuring that the zero slope lies in the set of the first condition of $\partial_{j} \mathcal{J}_{\lambda}$, i.e.
\begin{align*}
0 \in \left\{-\int_{\Theta} \left( u-\sum_{i} \hat u_i^F \varphi_i \right) \varphi_j f_{\Xi}\,d\xi + \lambda\eta \int_{\Theta} \left\vert L \varphi_j \right\vert f_{\Xi}\,d\xi : \eta \in [-1,1] \right\}.
\end{align*}
Using orthonormality and the definition of gPC coefficients gives
\begin{align*}
0 \in \left\{-\hat u_j +\hat u_j^F + \lambda\eta \int_{\Theta} \left\vert L \varphi_j \right\vert f_{\Xi}\,d\xi : \eta \in [-1,1] \right\}.
\end{align*}
Recalling that $\hat u_j^F=0$ and $L\varphi_j = -j(j+1)\varphi_j$ yields
\begin{align*}
0 \in \left\{-\hat u_j + \lambda\eta j(j+1)  \Vert  \varphi_j \Vert_{L^1} : \eta \in [-1,1] \right\}.
\end{align*}
Hence
\begin{align}\label{eq:f0condition}
\hat u_j \in [-\lambda j(j+1)  \Vert  \varphi_j \Vert_{L^1},\lambda j(j+1)  \Vert  \varphi_j \Vert_{L^1}],
\end{align}
which indicates for which values of $\hat u_j$ the filtered coefficient have a value of zero.
For non-zero values of $\hat u_j^F$, i.e. if
\begin{align}\label{eq:LassoCondition2}
\hat u_j \notin [-\lambda j(j+1)  \Vert  \varphi_j \Vert_{L^1},\lambda j(j+1)  \Vert  \varphi_j \Vert_{L^1}],
\end{align}
the gradient becomes the second condition of \eqref{eq:LassoFilterFunctionGrad}, which is
\begin{align}\label{eq:Grad0}
\partial_{j} \mathcal{J}_{\lambda}(\bm{\hat{u}}^F) =& -\int_{\Theta} \left( u-\sum_{i} \hat u_i^F \varphi_i \right)  \varphi_j f_{\Xi}\,d\xi + \lambda\text{sign}(\hat u_j^F) \int_{\Theta} \left\vert L  \varphi_j \right\vert f_{\Xi}\,d\xi \nonumber \\
=& -\int_{\Theta} u  \varphi_j f_{\Xi}\,d\xi -\sum_{i} \hat u_i^F\int_{\Theta}  \varphi_i  \varphi_j f_{\Xi}\,d\xi + \lambda j(j+1) \text{sign}(\hat u_j^F) \int_{\Theta} \left\vert  \varphi_j \right\vert f_{\Xi}\,d\xi \nonumber \\
=& -\hat u_j +\hat u_j^F + \lambda j(j+1) \text{sign}(\hat u_j^F) \Vert  \varphi_j \Vert_{L_1} \stackrel{!}{=}0.
\end{align}
To determine the sign of $\hat u_j^F$, we rearrange
\begin{align*}
\hat u_j =& \hat u_j^F + \lambda j(j+1) \text{sign}(\hat u_j^F) \Vert  \varphi_j \Vert_{L^1} \\
=& \hat u_j^F\underbrace{\left(1+\lambda j(j+1) \frac{1}{\vert \hat u_j^F \vert} \Vert  \varphi_j \Vert_{L^1}\right)}_{>0} \Rightarrow \text{sign}\left(\hat u_j\right) = \text{sign}\left(\hat u_j^F\right).
\end{align*}
Plugging this into \eqref{eq:Grad0} yields
\begin{align*}
\hat u_j^F =& \hat u_j - \lambda j(j+1) \text{sign}(\hat u_j) \Vert  \varphi_j \Vert_{L^1} \\
=&\hat u_j\left( 1 - \lambda j(j+1)\Vert  \varphi_j \Vert_{L^1}\frac{1}{\vert \hat u_j \vert}\right).
\end{align*}
Note that the case $1 - \lambda j(j+1)\Vert  \varphi_j \Vert_{L^1}\frac{1}{\vert \hat u_j \vert}<0$ does not occur, since 
\begin{align*}
1 - \frac{\lambda j(j+1)\Vert  \varphi_j \Vert_{L^1}}{\vert \hat u_j \vert}&<0 \\
\Leftrightarrow\vert \hat u_j \vert - \lambda j(j+1)\Vert  \varphi_j \Vert_{L^1}&<0 \\
\Leftrightarrow\vert \hat u_j \vert &< \lambda j(j+1)\Vert  \varphi_j \Vert_{L^1}
\end{align*}
holds, meaning that 
\begin{align*}
\hat u_j\in [-\lambda j(j+1)\Vert  \varphi_j \Vert_{L^1},\lambda j(j+1)\Vert  \varphi_j \Vert_{L^1}],
\end{align*}
which violates \eqref{eq:LassoCondition2}. Hence, in this case, we need to look at the first condition of the subdifferential, meaning that $\hat u_j^F$ must be set to zero.
Using the notation
\begin{align*}
x_+ = 
\begin{cases}
x \enskip \text{ if } x>0 \\
0 \enskip \text{ else}
\end{cases},
\end{align*}
the filtered coefficient can be written as
\begin{align*}
\hat u_j^F = \hat u_j\left( 1 - \frac{\lambda j(j+1)\Vert  \varphi_j \Vert_{L^1}}{\vert \hat u_j \vert}\right)_+,
\end{align*}
which yields the filter function from the theorem.
\end{proof}

The constructed filter will in the following be called \textit{Lasso filter}.  In contrast to standard filters, the filter function \eqref{eq:LassoFilterFunction} depends on the moments of the solution. The first moment is not modified while higher order moments are dampened. Note that if
\begin{align*}
\frac{\lambda j(j+1)\Vert  \varphi_j \Vert_{L^1}}{\vert \hat u_j \vert} \geq 1,
\end{align*}
i.e. when the order of the moment increases or the absolute value of the moment decreases, the filtered moment will be chosen to be zero.

\begin{figure}[h!]
\centering
\begin{subfigure}[b]{0.8\textwidth}
   \includegraphics[width=1\linewidth]{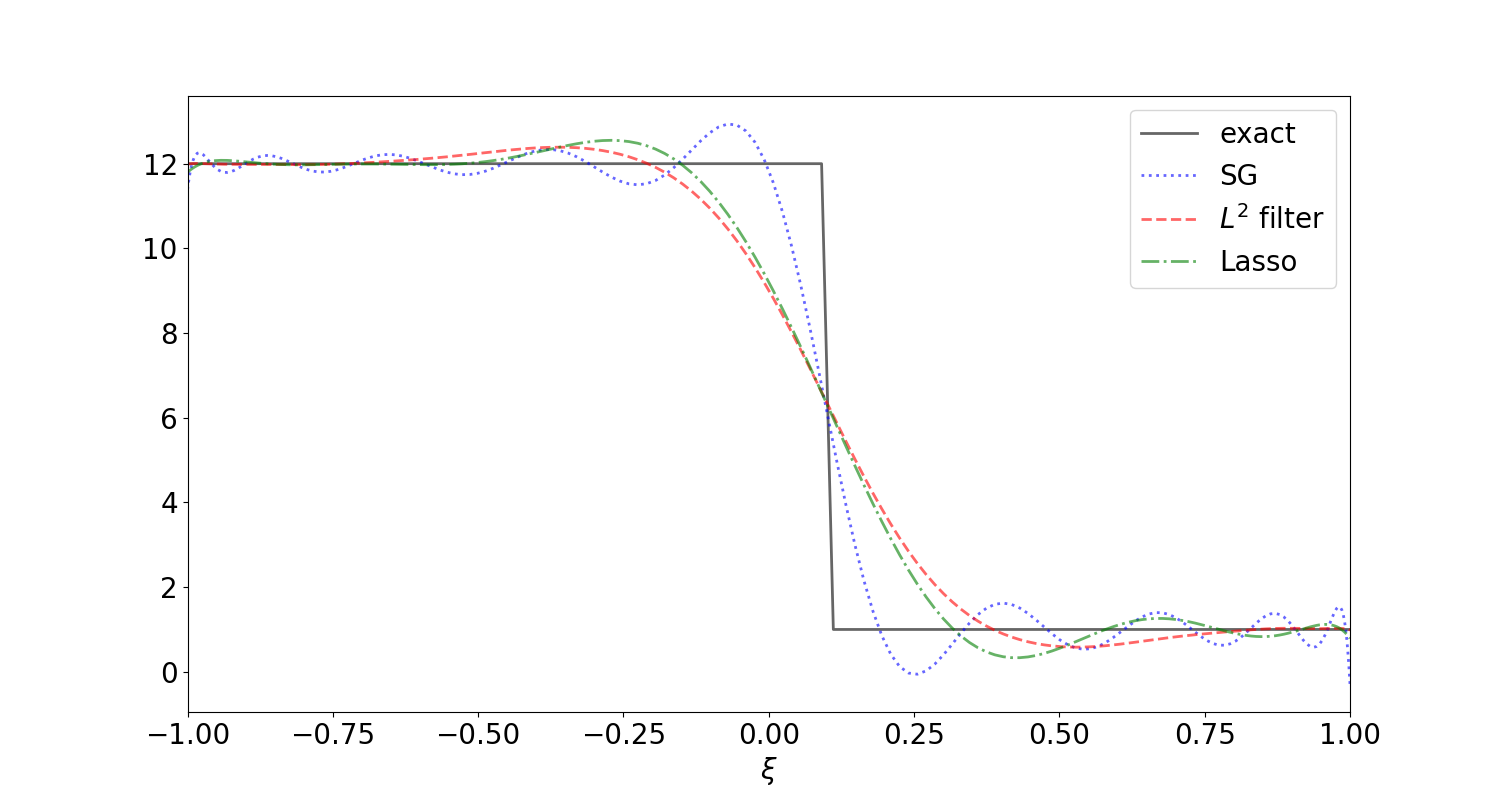}
   \caption{}
   \label{fig:solutionApproximationFilter} 
\end{subfigure}

\begin{subfigure}[b]{0.8\textwidth}
   \includegraphics[width=1\linewidth]{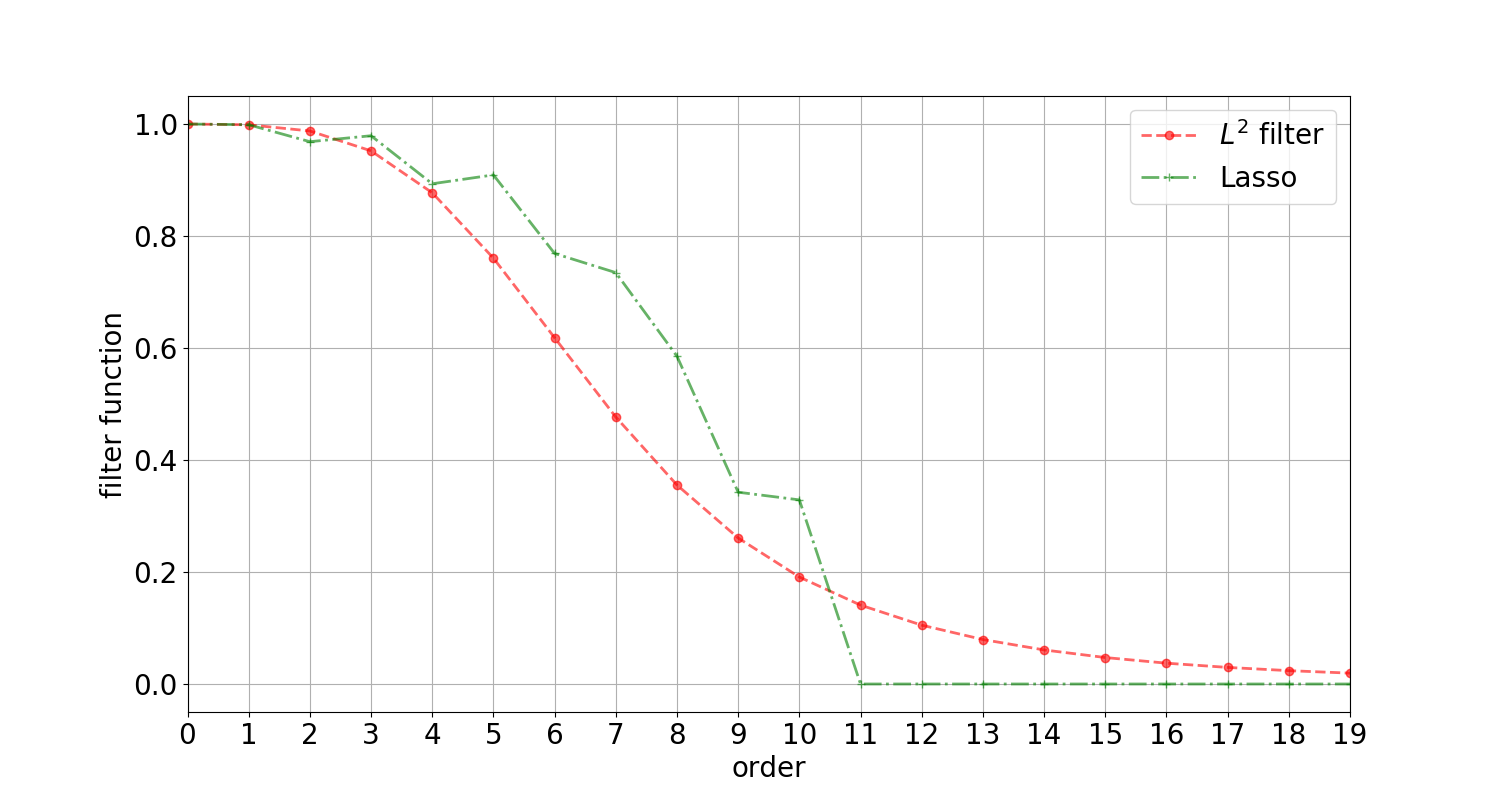}
   \caption{}
   \label{fig:filterFunction}
\end{subfigure}

\caption{(a) Approximation of a shock using SG and filtered SG with $20$ expansion coefficients. (b) Filter functions for Lasso with $\lambda = 0.0035$ and L$^2$ filtering with $\lambda = 0.00035$.}
\end{figure}

To demonstrate the effects of filtering, we consider a shock as depicted in Figure~\ref{fig:solutionApproximationFilter}. Using the standard SG approximation \eqref{eq:gPC} leads to oscillatory solutions. The filtered approximations are computed by \eqref{eq:FilteredGPC} making use of the filtered gPC coefficients \eqref{eq:L2FilterFunction} with the filter functions \eqref{eq:LassoFilterFunction} for the Lasso filter and \eqref{eq:L2FilterFunction} for the L$^2$ filter. Filter parameters for both methods are chosen such that the solution approximations show similar behavior. The corresponding filter functions are shown in Figure~\ref{fig:filterFunction}. In contrast to the L$^2$ filter, the Lasso filter yields a sparse solution representation, due to the fact that all coefficients with degree bigger than 10 are set to zero. The L$^2$ filter keeps coefficients with high order, however their contribution to the solution approximation is negligible. Effects of using filter functions are shown in Figure~\ref{fig:solutionApproximationFilter}, where both, the Lasso and L$^2$ filtered solution mitigate oscillations at the cost of resulting in a smeared out shock approximation. 

When the exact solution does not depend on $\xi$, i.e. only the zeroth order moment is non-zero, the filter does not affect the approximation since $g(0) = 1$. Consequently, the filters allow a sharp approximation of a deterministic shock in the spatial domain.

Note, that we have only discussed scalar random variables. For multi-dimensional problems, i.e. if $\bm{\xi}\in\mathbb{R}^{K}$ with $K>1$, the operator $L$, which punishes oscillations can be applied multiple times: If
\begin{align*}
L_k u(\bm\xi):=\partial_{\xi_k}((1-\xi_k^2)\partial_{\xi_k} u(\bm\xi)),
\end{align*}
one can use the operator
\begin{align*}
L u(\bm\xi) := (L_1 \circ L_2 \circ \cdots \circ L_K)(u(\bm\xi))
\end{align*}
in the cost function of the optimization problem \eqref{eq:costFunction} or for the Lasso filter \eqref{eq:LassoFunctional}.

Before turning to the choice of the filter strength $\lambda$, we need to discuss how the filtering procedure can be integrated into the SG framework. The idea is to replace the standard gPC coefficients in the time update of the numerical discretization by the filtered coefficients. Details on how the filter can be integrated into a given SG solver are discussed in the following.

\section{Numerical implementation}
\label{sec:numerics}

In this section, we present the numerical discretization of the moment system, when using filters. Here, we again assume a non-scalar problem, meaning that $p>1$. To discretize the moment system in time and space, a spatial grid $x_1,\cdots,x_{N_x}$ and a time grid $t_0,\cdots,t_{N_t}$ is used. The discretized solution is then given by 
\begin{align}\label{eq:discreteMoments}
\hat u_{sij}^n \simeq \frac{1}{\Delta x}\int_{x_{j-1/ 2}}^{x_{j+1/ 2}}\hat u_{si}(t_n,x) dx,
\end{align}
i.e. $s$ indicates the state, $i$ is the moment order and $j$ is the spatial cell. In each cell at every time step, the moment vector is collected in $\bm{\hat u}_j^n = (\hat u_{sij}^n)_{s,i}$. If a numerical flux $\bm{f^*}(\bm{u}_{\ell},\bm{u}_r)$ for the initial hyperbolic equation \eqref{eq:hyperbolicProblem} is given, a flux for the moment system can be constructed by
\begin{align}\label{eq:numFlux}
\bm{F}^*(\bm{\bm{\hat u}_{j}^n,\bm{\hat u}_{j+1}^n}) = \int_{\Theta} \bm{f^*}(\bm{\hat u}_{j}^n\bm{\varphi},\bm{\hat u}_{j+1}^n\bm{\varphi}) \bm{\varphi} f_{\Xi} \,d\xi.
\end{align}
Note that in general, this integral can be computed analytically. After having computed the moments of the initial condition via \eqref{eq:discreteMoments}, the moment vector can be time-updated iteratively by
\begin{align}
\bm{\hat u}_j^{n+1} = \bm{\hat u}_j^{n} -\frac{\Delta t}{\Delta x}(\bm{F^*}(\bm{\hat u}_{j}^n,\bm{\hat u}_{j+1}^n)-\bm{F^*}(\bm{\hat u}_{j-1}^n,\bm{\hat u}_{j}^n))
\end{align}
This gives a numerical approximation of the standard SG moment system \eqref{eq:SGMomentSystem}. To dampen oscillations, the filtering step is included in every time step before the moments are used in the numerical fluxes. The filtered SG scheme can be found in algorithm \ref{alg:fSG}.

\begin{algorithm}[H]
\begin{algorithmic}[1]
\STATE$\bm{\hat u}^{0}_j \leftarrow setupInitialConditions$ for all cells $j$
\STATE choose $\lambda$
\FOR{$n=0$ to $N_t$}

\FOR{$s = 1$ to $p$, $i = 0$ to $N$, $j = 1$ to $N_x$}
\STATE $\hat u_{sij}^n \leftarrow g_{\lambda}(i,\hat u_{sij}^n)\hat u_{sij}^n$
\ENDFOR

\FOR{$j=1$ to $N_x$}
\STATE $ \bm{\hat u}_j^{n+1} \leftarrow \bm{\hat u}_j^{n} -\frac{\Delta t}{\Delta x}(\bm{F^*}(\bm{\hat u}_{j}^n,\bm{\hat u}_{j+1}^n)-\bm{F^*}(\bm{\hat u}_{j-1}^n,\bm{\hat u}_{j}^n))$ 
\ENDFOR

\ENDFOR
\end{algorithmic}
\caption{Filtered Stochastic Galerkin Method}
\label{alg:fSG}
\end{algorithm}

\section{Choosing the filter strength}
\label{sec:filterStrength}
A cumbersome task when using filters is to select an adequate filter strength $\lambda$, which sufficiently damps oscillations while preserving general characteristics of the exact solution. Since the optimal filter strength is problem dependent, a parameter study must be conducted for finding an optimal value for $\lambda$. Furthermore, the filter strength does not depend on the solution, i.e. smooth regions are as strongly dampened as discontinuities. In the following, an automated procedure to pick an adequate filter strength is proposed. The resulting filter is different for every spatial cell as well as every time step.\\

We start writing down the Lasso optimization problem \eqref{eq:LassoFunctional} for a given truncation order $M$:
\begin{align*}
\mathcal{J}_{\lambda}(\bm{\hat{u}}^F) = \frac{1}{2}\int_{\Theta} \left( u-\sum_{i=0}^M \hat{u}_i^F \varphi_i \right)^2 f_{\Xi}\,d\xi + \lambda \int_{\Theta}\sum_{i=0}^M \left\vert L\hat{u}_i^F \varphi_i \right\vert f_{\Xi}\,d\xi.
\end{align*}
Without the Lasso regression term, the solution of the optimization problem will yield the exact solution $u$ for $M\to\infty$. When adding the Lasso term with some choice for $\lambda$, we observe that the solution to the optimization problem becomes sparse and for some $\tilde N$, all moments $\hat u_i$ with $i>\tilde N$ are set to zero. Thus solving the SG system with the truncation order $\tilde N$ or with a much higher order $M \gg \tilde N$, where $M$ can even be infinite, will yield the same result.

Keeping this observation in mind, we have two options and are in zugzwang: 1) either make some choice for $\lambda$ which then tells us a suitable truncation order $\tilde N$, or 2) pick a truncation order, which then determines the filtering coefficient $\lambda$. In this work, we choose option 2, i.e.\ we derive a strategy to choose an adequate $\lambda$ for a given truncation order: Denoting this truncation order by $N$, the filter strength of state $s$ in cell $j$ at time step $n$ is given by
\begin{align}
\lambda^* = \frac{\vert \hat u_{N,j}^n\vert}{ N(N+1) \Vert \varphi_N \Vert_{L^1}}.
\end{align}
This choice ensures that the $N_{th}$ filtered coefficient is zero. Since the filter function $g(i,\hat u_i)$ decreases quadratically in $i$, the event that all moments $\hat u_i$ with $i>N$ in the individual cell at the given time are zero is likely. Therefore, with this choice of the filter coefficient, we obtain the same solution as with an order $M\gg N$ moment system. The resulting filtering function is then given by
\begin{align}
\tilde{g}\left(\hat u_{i,j}^n ,\hat u_{N,j}^n\right)=\left(1 - \frac{ i(i+1)}{N(N+1)}\frac{\Vert  \varphi_i \Vert_{L^1}}{ \Vert \varphi_N \Vert_{L^1}}\frac{\vert \hat u_{N,j}^n\vert}{\vert \hat u_{i,j}^n \vert} \right)_+.
\end{align}
In the system case, we have
\begin{align}
\tilde{g}\left(\hat u_{s,i,j}^n ,\hat u_{s,N,j}^n\right)=\left(1 - \frac{ i(i+1)}{N(N+1)}\frac{\Vert  \varphi_i \Vert_{L^1}}{ \Vert \varphi_N \Vert_{L^1}}\frac{\vert \hat u_{s,N,j}^n\vert}{\vert \hat u_{s,i,j}^n \vert} \right)_+.
\end{align}
Substituting this choice of the filter coefficient into algorithm~\ref{alg:fSG} yields the method we use in the following section to obtain non-oscillatory approximations of expected value and variance.

\section{Numerical Results}
\label{sec:results}

In the following, we compare the Lasso filter to standard SG and IPM. The IPM method can yield so-called non-realizable moments that lead to a failure of the optimization problem. To prevent this behavior, we recalculate moments with dual states from the optimization problem as discussed in \cite{kusch2017maximum}.

\subsection{Burgers' equation}
In the following, we study the stochastic Burgers' equation, which reads
\begin{subequations}\label{eq:Burgers}
\begin{align}
\partial_t &u(t,x,\xi)+\partial_x \frac{u(t,x,\xi)^2}{2} = 0,\\
&u(t=0,x,\xi) = u_{\text{IC}}(x,\xi).
\end{align}
\end{subequations}
Following \cite{poette2009uncertainty}, we choose the random initial condition as
\begin{align}\label{eq:IC1}
u_{\text{IC}}(x,\xi) &:= 
\begin{cases} u_L, & \mbox{if } x< x_0+\sigma\xi \\ u_L+\frac{u_R-u_L}{x_0-x_1} (x_0+\sigma \xi-x), & \mbox{if } x\in[x_0+\sigma \xi,x_1+\sigma \xi]\\
u_R, & \text{else }
\end{cases}
\end{align}
which is a forming shock with a linear connection from $x_0+\sigma \xi$ to $x_1+\sigma \xi$. The random variable $\xi$ is uniformly distributed on the interval $\Theta = [-1,1]$. Furthermore, we have Dirichlet boundary conditions $u(t,a,\xi) = u_L$ and $u(t,b,\xi)=u_R$. The numerical flux is chosen according to \eqref{eq:numFlux} where the underlying numerical flux $f^*$ is chosen to be Lax-Friedrichs. Additionally, we use the following parameter values:\\
\begin{center}
    \begin{tabular}{ | l | p{7cm} |}
    \hline
    $[a,b]=[0,3]$ & range of spatial domain \\
    $N_x=2000$ & number of spatial cells \\
    $t_\mathrm{end}=0.11$ & end time \\
    $x_0 = 0.5, x_1=1.5, u_L = 12, u_R = 1, \sigma = 0.2$ & parameters of initial condition \eqref{eq:IC1} \\
    \hline
    \end{tabular}
\end{center}

\begin{figure}[h!]
\centering
\begin{subfigure}{.5\textwidth}
  \centering
  \includegraphics[width=1.0\linewidth]{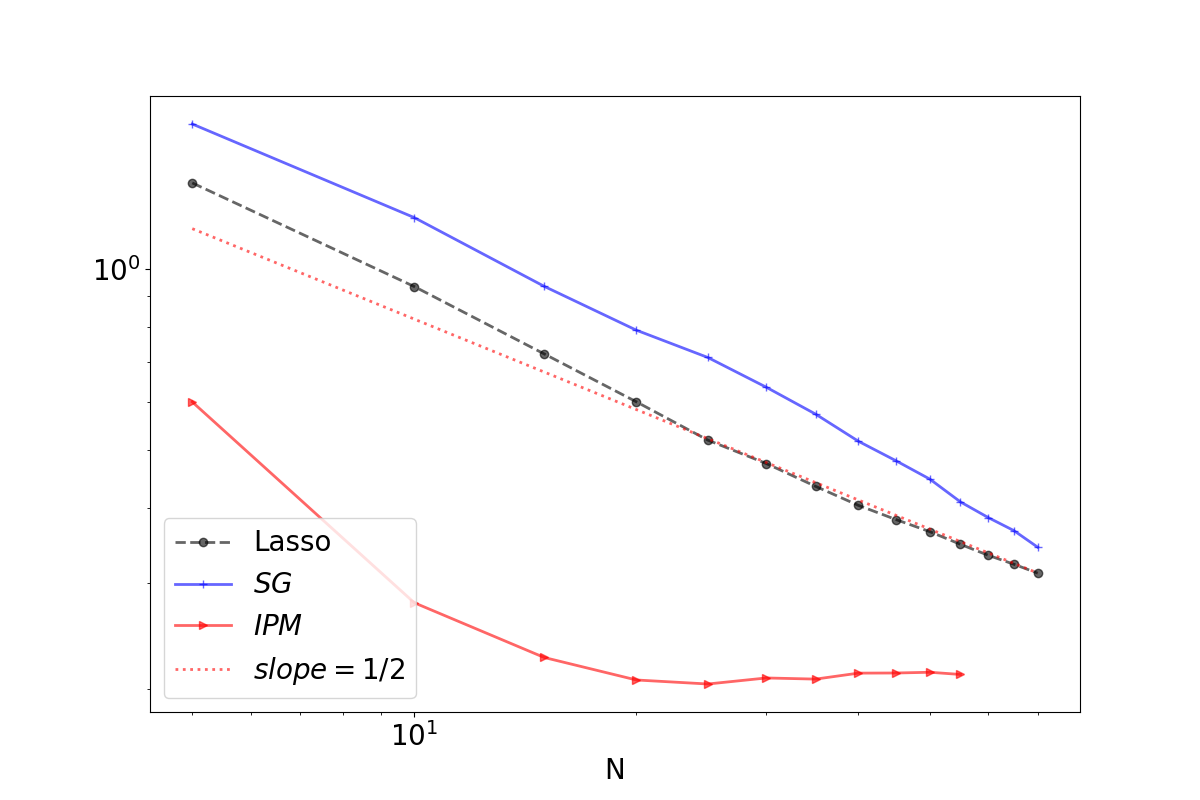}
  \caption{Error of solution.}
  \label{fig:ConvergenceL2Error}
\end{subfigure}%
\begin{subfigure}{.5\textwidth}
  \centering
  \includegraphics[width=1.0\linewidth]{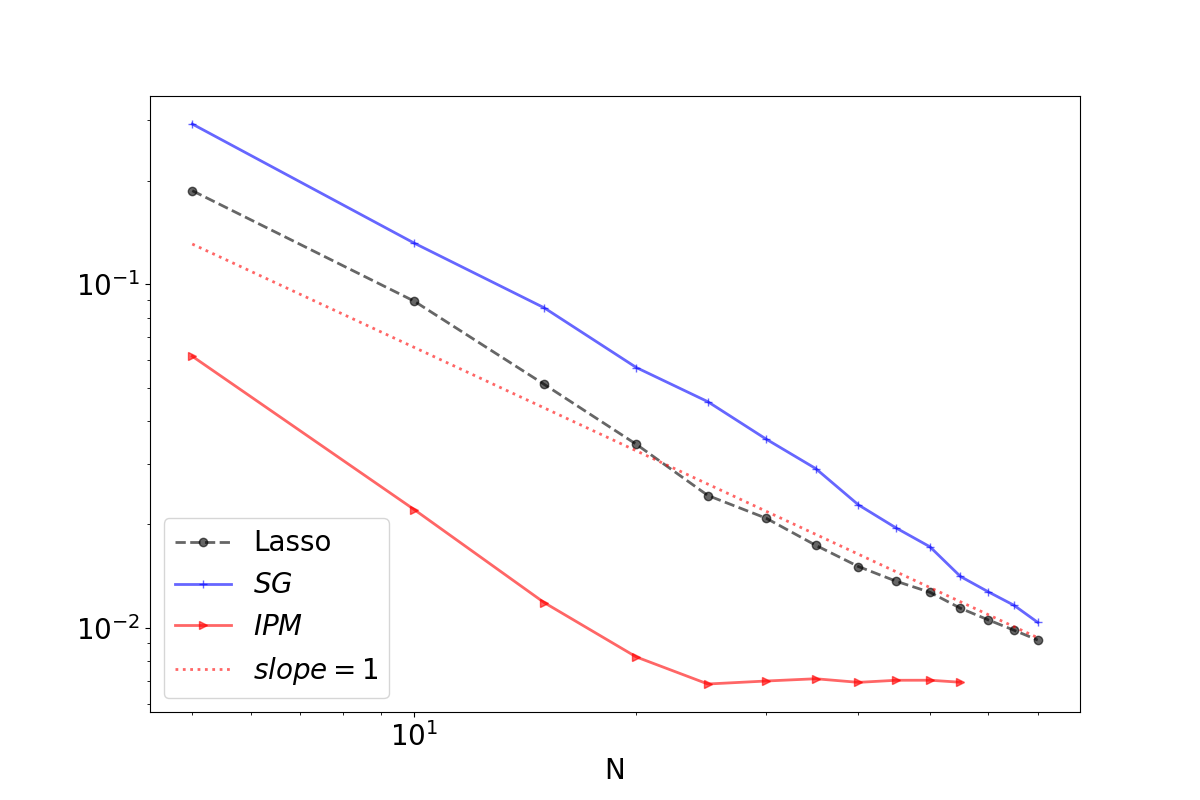}
  \caption{Error of expectation value.}
  \label{fig:ConvergenceExpectedValue}
\end{subfigure}
\caption{Convergence behavior of Burgers' equation for increasing truncation order $N$.}
\label{fig:Convergence}
\end{figure}

\begin{figure}[h!]
\centering
  \includegraphics[width=0.8\linewidth]{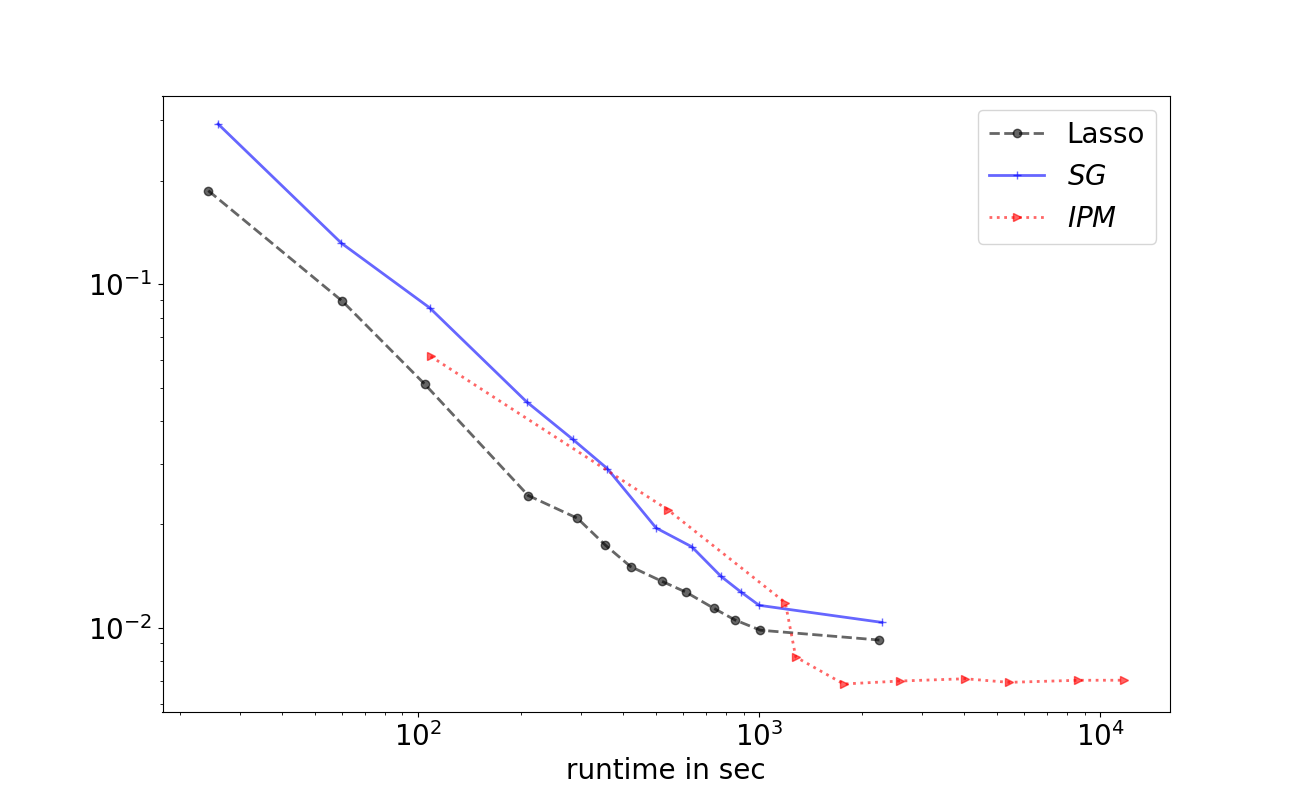}
  \caption{Error of the expectation value plotted over runtime.}
  \label{fig:efficencyPlotBurgers}
\end{figure}

For the IPM solution, we use the bounded--barrier entropy introduced in \cite{kusch2017maximum}, which is
\begin{align}\label{eq:BBEntropy}
s(u) = (u-u_-)\ln(u-u_-)+(u_+-u)\ln(u_+-u),
\end{align}
where $u_- = \min_{x,\xi} u_{\text{IC}}$ and $u_+ = \max_{x,\xi} u_{\text{IC}}$.  Note that when using SG as well as the Lasso filter, all arising integrals can be computed exactly before running the program. In the case of IPM, we use a Gauss-Lobatto quadrature rule with $4N$ quadrature points. With increasing time, the exact solution forms a shock in the random as well as physical space. The effects of the SG and filtered SG solution, when approximating this discontinuity are studied in the following. We increase the number of moments and observe the resulting error of the solution as well as the expectation value. The error of the solution itself, i.e.
\begin{align*}
\Vert u - u_N \Vert_{L^2(\mathcal{D},\Theta)} := \sqrt{\int_{\mathcal{D}} \int_{\Theta} \left(u(t_{\text{end}},x,\xi)-u_N(t_{\text{end}},x,\xi)\right)^2 f_{\Xi}\,d\xi dx}
\end{align*}
is shown in Figure \ref{fig:ConvergenceL2Error}. Note that the convergence does not only depend on the projection error which is $1/2$, but also on the method's closure error (i.e. the error arising from the error in the physical flux due to the approximation). Both methods show an overall convergence speed of $1/2$, however the filtered SG starts at a smaller error value. Consequently, the filtered Solution computed with $20$ moments has a smaller error than the classical SG solution with $30$ moments. When the moment number increases, the SG result approaches the filtered SG solution. This is due to the fact that the last moments is getting close to zero, i.e. the filter is turned off. The IPM method gives a good approximation, already for a small moment order. After a truncation order of $15$, the error is not further decreased, which is most likely caused by a too dominant error of the spatial and time discretization. A similar behavior can be found in Figure \ref{fig:ConvergenceExpectedValue}. Here, the errors of the expectation value are plotted for different numbers of moments. Since the expectation value is smoother than the solution, we expect a faster convergence to the exact solution. The order of convergence appears to be in the order of one. Again, the filtered SG yields a smaller error and is approached by the SG error values for increasing truncation order $N$. The IPM yields a smaller error for the expectation value. Compared to the convergence of the solution, the error of the expectation value decreases until $25$ moments, after which the discretization error dominates the overall error in the solution.\\

However, a main challenge of IPM is its increased numerical costs. Approaches to circumvent this are efficient high order numerical schemes for the spatial and time discretization \cite{kusch2017maximum} as well as parallelization \cite{garrett2015optimization}. In the following, we compare the resulting error of the expectation value for a given runtime in Figure \ref{fig:efficencyPlotBurgers}. All three methods are run on a desktop computer without parallelization. It can be seen that the efficiency curve of the Lasso filter lies below the other methods for most runtimes, i.e. the resulting error is the smallest for a given runtime. The IPM lies below the Lasso curve for very long computation times. 

\begin{figure}[h!]
\centering
  \includegraphics[width=0.8\linewidth]{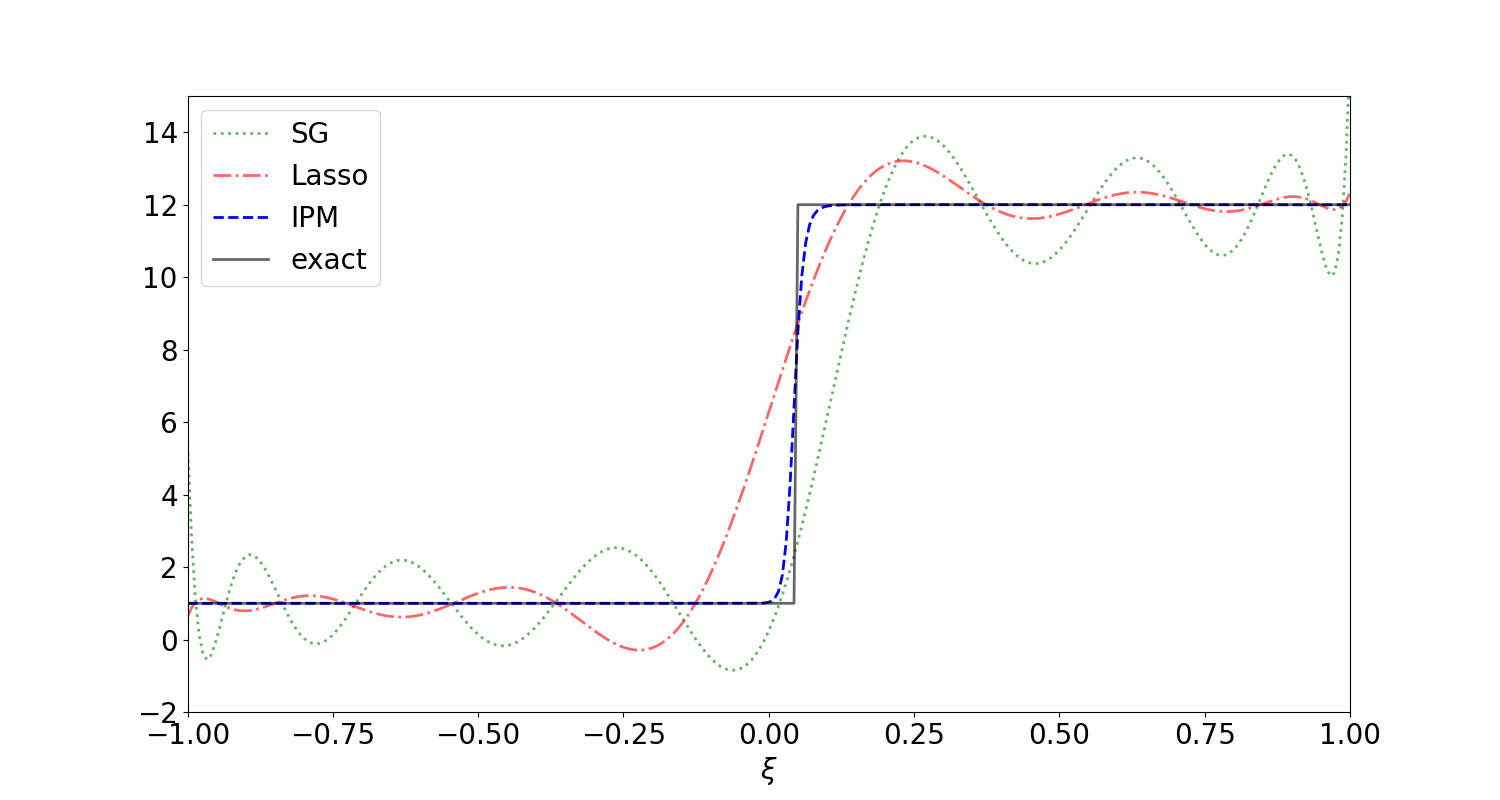}
  \caption{Solutions for fixed spatial position $x^*$ for Burgers equation.}
  \label{fig:IC1fixedX}
\end{figure}

\begin{figure}[h!]
\centering
  \includegraphics[width=0.8\linewidth]{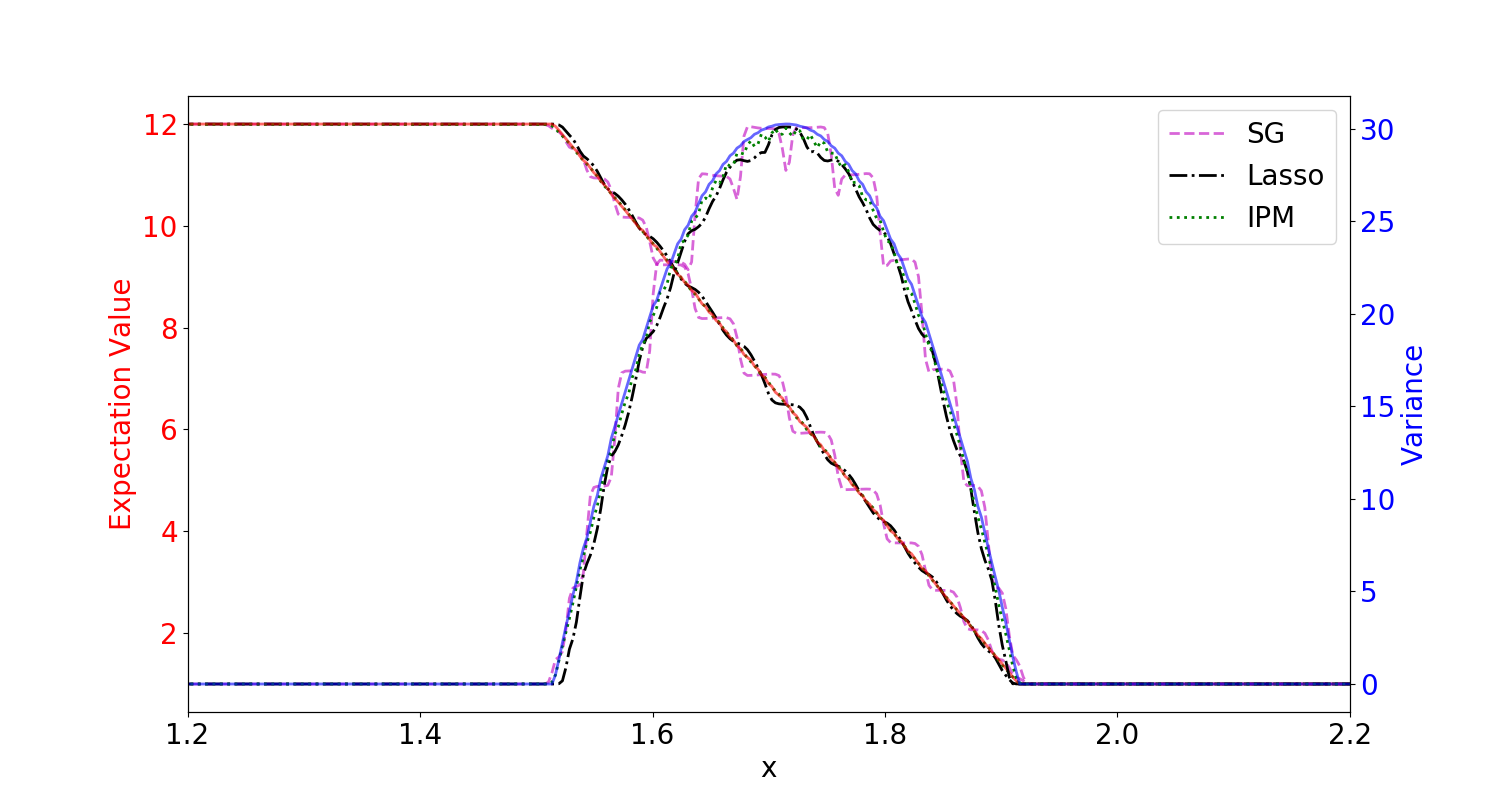}
  \caption{Expectation value and variance for Burgers' equation. The exact expectation value is depicted by the red and the exact variance by the blue line.}
  \label{fig:IC1ExpVar}
\end{figure}

We now take a look at the solution of the stochastic Burgers' equation for a fixed truncation order $N=15$ and compare the results of SG, Lasso and IPM.   Since the moment system of IPM cannot be integrated analytically, we make use of a sixty-point Gauss-Legendre quadrature. The comparison of all three methods for a fixed spatial position $x^* = 1.72$ is shown in Figure \ref{fig:IC1fixedX}. We can see that IPM yields a well-resolved solution approximation, which fulfills the maximum principle. Both SG and filtered SG violate the maximum principle, however come at a cheaper numerical cost. Compared to SG, the filtered SG shows dampened oscillations and a better capturing of the shock position. Note that the polynomial order of the filtered SG is $14$ instead of $15$, since the last moment always has a value of zero by the construction of the filter strength. Taking a look at the comparison of the expectation value in Figure \ref{fig:IC1ExpVar}, we can see that the SG result shows a step-like profile, thus yielding a non-satisfactory solution approximation. The IPM and the filtered SG can approximate the exact solution nicely. Note that Lasso shows a small step in the middle. Taking a look at the variance, we see that SG yields an oscillatory result. The variance computed with IPM lies closer to the exact variance than the variance coming from Lasso, however both methods yield a satisfactory approximation. 

\subsection{Euler 1D}
Though we have managed to avoid an opening loss, in solving a scalar, hyperbolic problem, the true utility of this method will require demonstration on hyperbolic systems and multi-dimensional problems.
In the following, we investigate the stochastic Euler equations in one spatial dimension before solving problems in higher spatial dimensions. In contrast to SG and filtered SG, the IPM moment system remains hyperbolic, i.e. density and pressure remain positive. However, methods to enforce hyperbolicity of the SG moment system such as hyperbolicity limiters \cite{schlachter2017hyperbolicity} exist and can be combined with filters. Test cases discussed in this paper remain hyperbolic without hyperbolicity limiters. The Euler equations are
\begin{align*}
\partial_t
\begin{pmatrix}
\rho \\ \rho u \\ \rho e
\end{pmatrix}
+\partial_x
\begin{pmatrix}
\rho u \\ \rho u^2 +p \\ \rho u (e+p)
\end{pmatrix}
=\bm{0},
\end{align*}
with the initial conditions
\begin{align*}
\rho_{\text{IC}} &= \begin{cases} \rho_L &\mbox{if } x < x_{\text{interface}}(\xi) \\
\rho_R & \mbox{else } \end{cases} \\
(\rho u)_{\text{IC}} &= 0 \\
(\rho e)_{\text{IC}} &= \begin{cases} \rho_L e_L &\mbox{if } x < x_{\text{interface}}(\xi) \\
\rho_R e_R & \mbox{else } \end{cases}
\end{align*}
Here, $\rho$ is the density, $u$ is the velocity and $e$ is the specific total energy. The pressure $p$ can be determined from
\begin{align*}
p = (\gamma-1)\rho\left(e-\frac{1}{2}u^2\right).
\end{align*}
The heat capacity ratio $\gamma$ has a value of $1.4$ for air. Due to the random interface position $x_{\text{interface}}(\xi) = x_0+\sigma \xi$, the solution is uncertain. Again, we use a uniformly distributed random variable with $\Theta=[-1,1]$. Similar to the Burgers' test case, Dirichlet boundary conditions are chosen at the left and right boundary. The underlying numerical flux $\bm{f}^*$ is the HLL-flux. We use the following parameter values:

\begin{center}
    \begin{tabular}{ | l | p{7cm} |}
    \hline
    $[a,b]=[0,1]$ & range of spatial domain \\
    $N_x=2000$ & number of spatial cells \\
    $t_\mathrm{end}=0.14$ & end time \\
    $x_0 = 0.5, \sigma = 0.05$ & interface position parameters\\
    $\rho_L,p_L = 1.0, \rho_R,p_R = 0.3$ & initial states\\
    $N = 15$ & polynomial degree \\
    $\tau = 10^{-7}$ & gradient tolerance for IPM \\
    \hline
    \end{tabular}
\end{center}

Note that for small densities, the moment system of SG and Lasso looses hyperbolicity, which is not the case for IPM. However, IPM comes at a highly increased numerical cost, since an optimization problem with $3(N+1)$ unknowns needs to be solved in all $2000$ spatial cells in every time step. While SG and Lasso take 90.8 seconds to compute, IPM runs for over five hours.

\begin{figure}[h!]
\centering
  \includegraphics[width=0.8\linewidth]{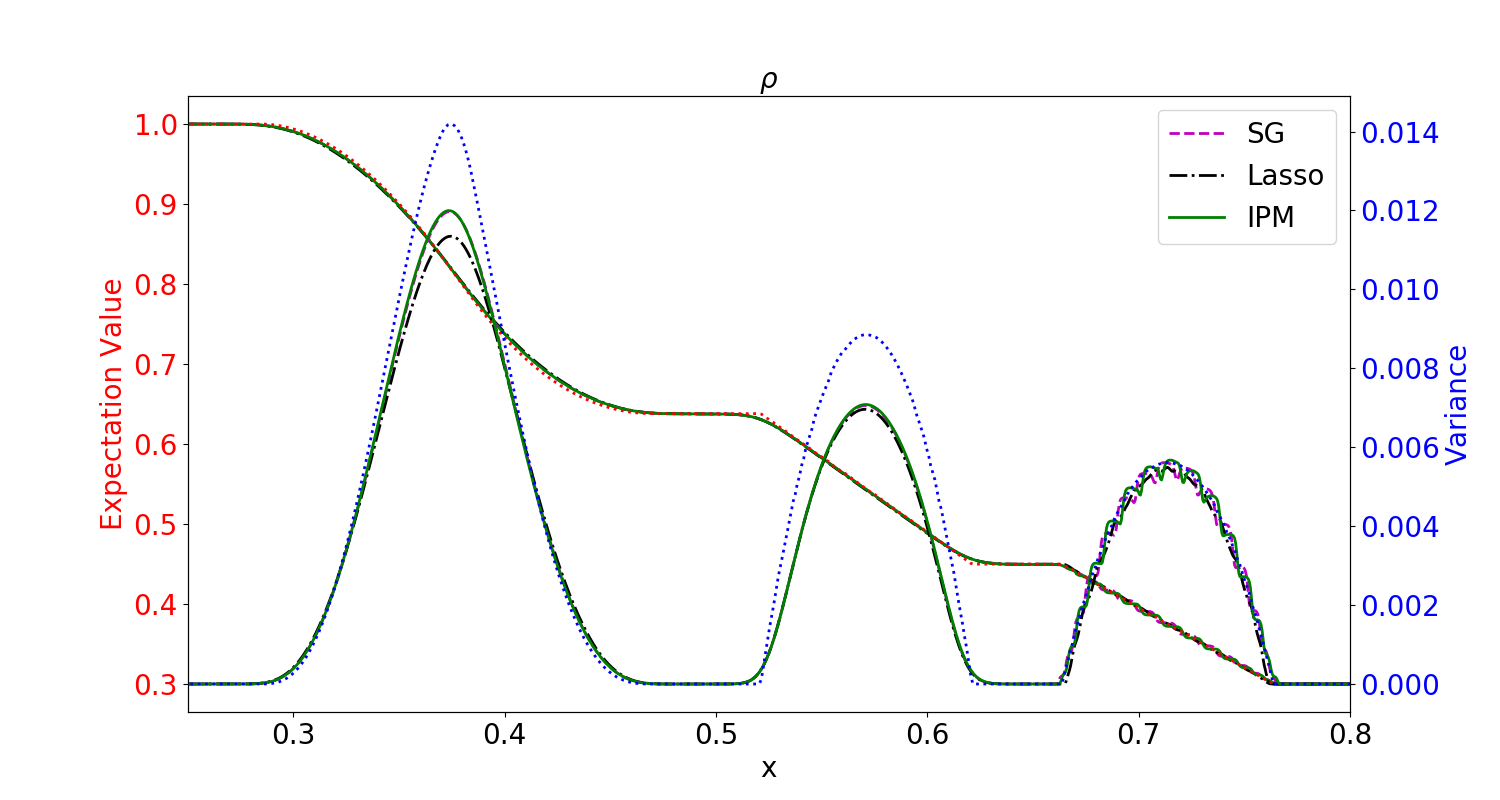}
  \caption{Expected value and variance of density. The exact expectation value is given by the red dotted line and the exact variance is given by the blue dotted line.}
  \label{fig:ExpVarRho}
\end{figure}

\begin{figure}[h!]
\centering
\begin{subfigure}{.5\textwidth}
  \centering
  \includegraphics[width=1.0\linewidth]{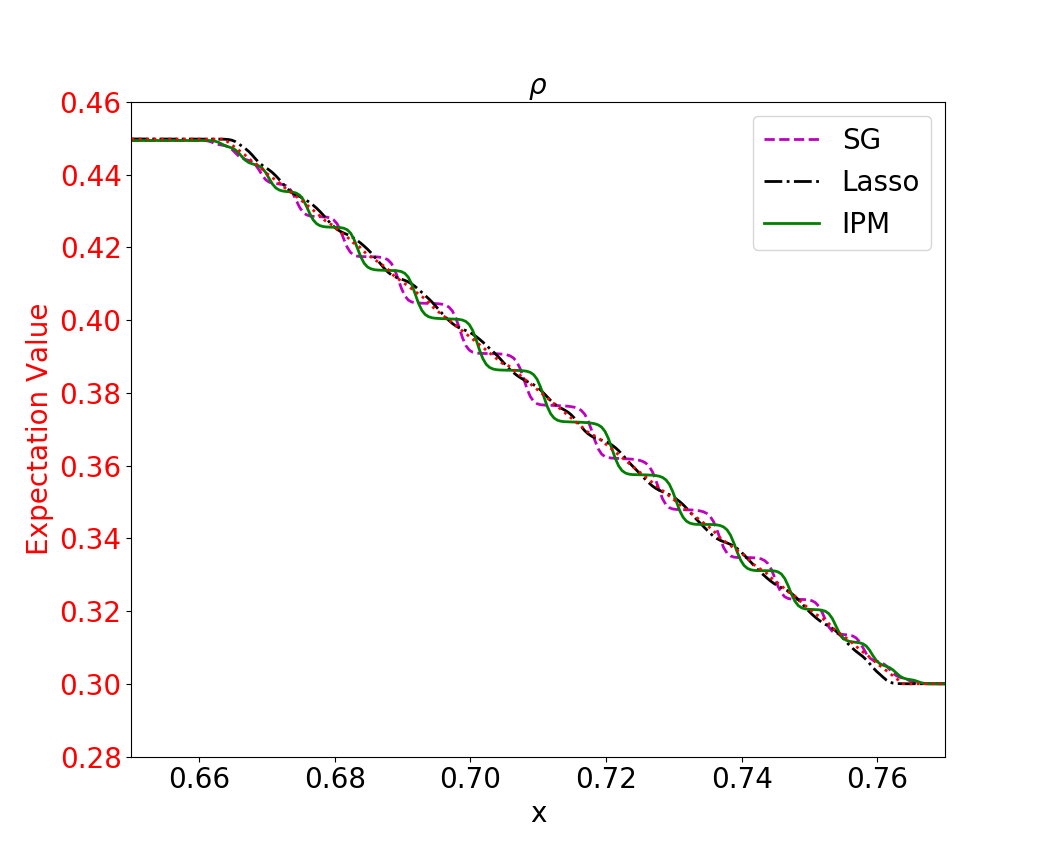}
  \caption{Zoomed view of shock.}
  \label{fig:ExpVarRhoShock1}
\end{subfigure}%
\begin{subfigure}{.5\textwidth}
  \centering
  \includegraphics[width=1.0\linewidth]{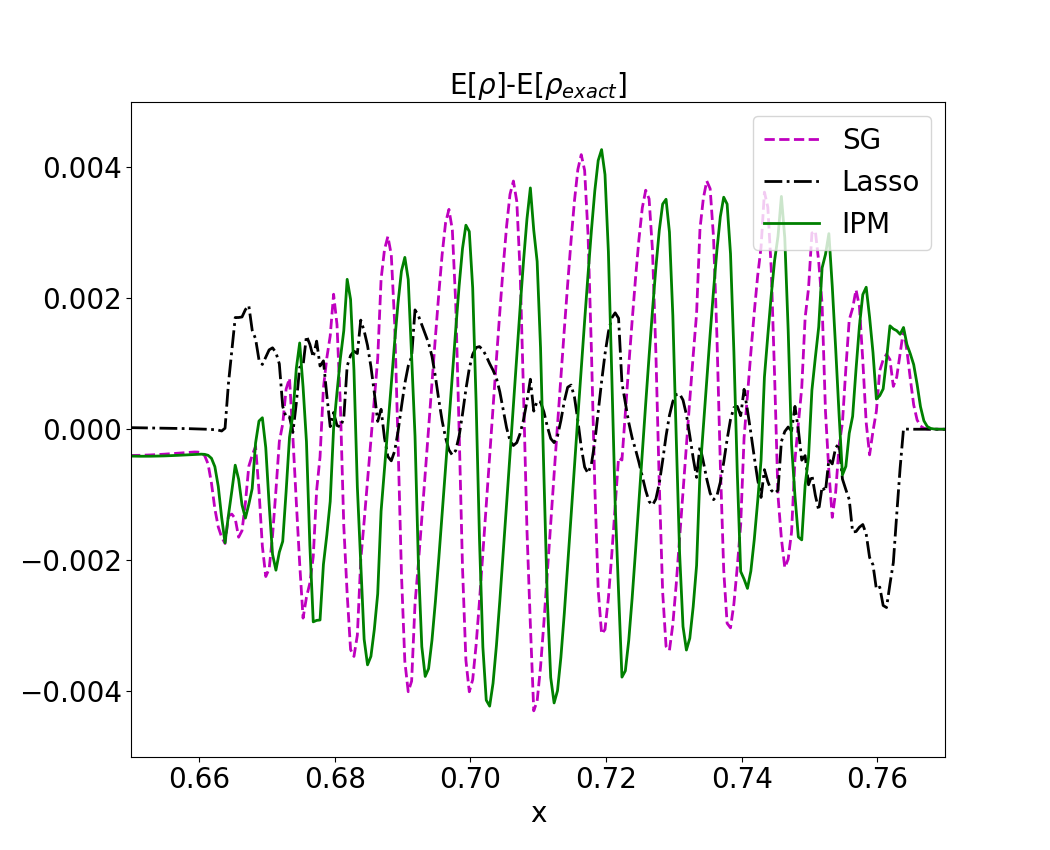}
  \caption{Distance to exact solution.}
  \label{fig:ExpVarRhoShock2}
\end{subfigure}
\caption{Expected value and difference to exact solution of the density shock.}
\label{fig:ExpVarRhoShock}
\end{figure}

\begin{figure}[h!]
\centering
  \includegraphics[width=0.8\linewidth]{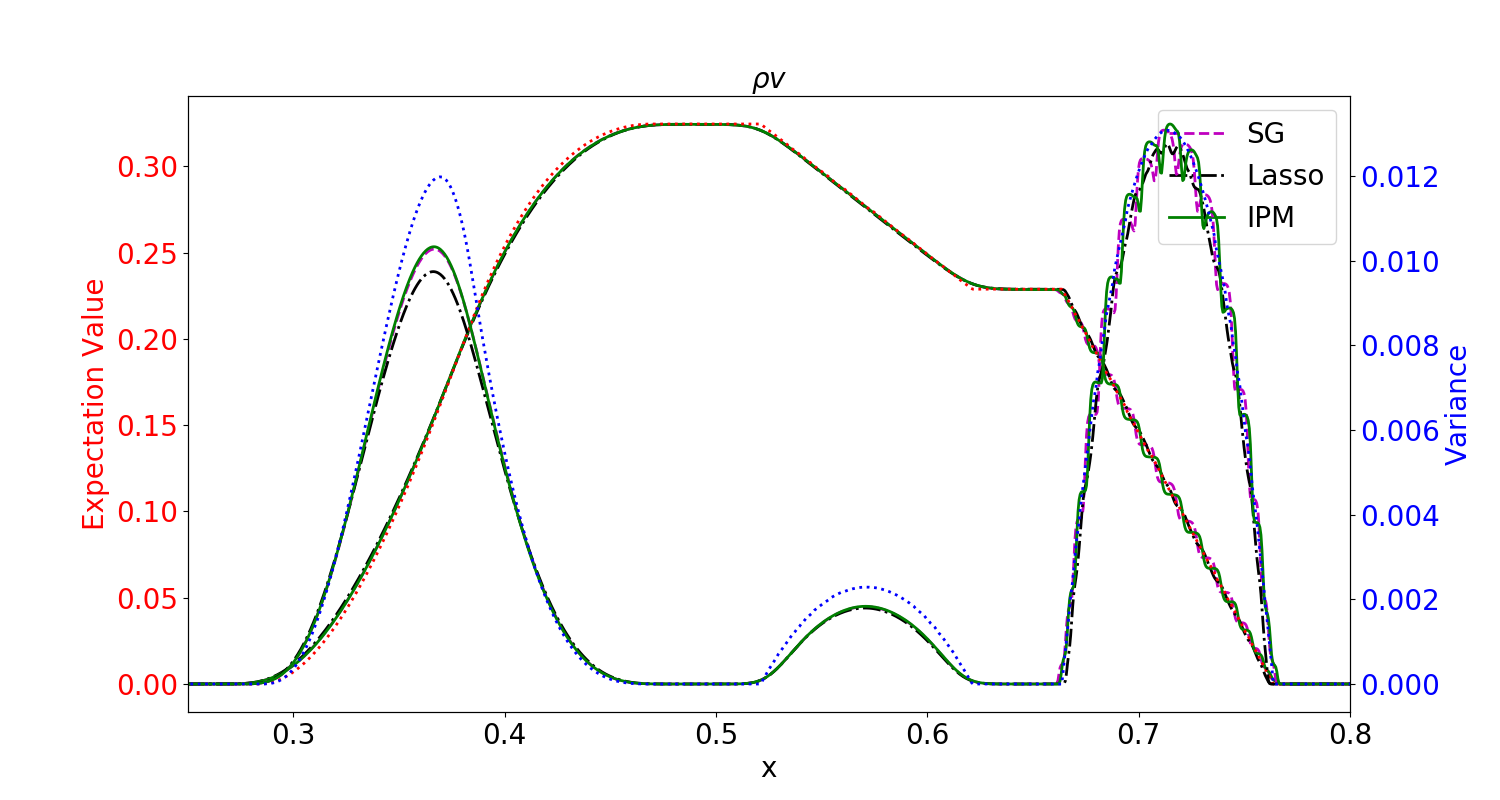}
  \caption{Expected value and variance of momentum.}
  \label{fig:ExpVarMoment}
\end{figure}

\begin{figure}[h!]
\centering
  \includegraphics[width=0.8\linewidth]{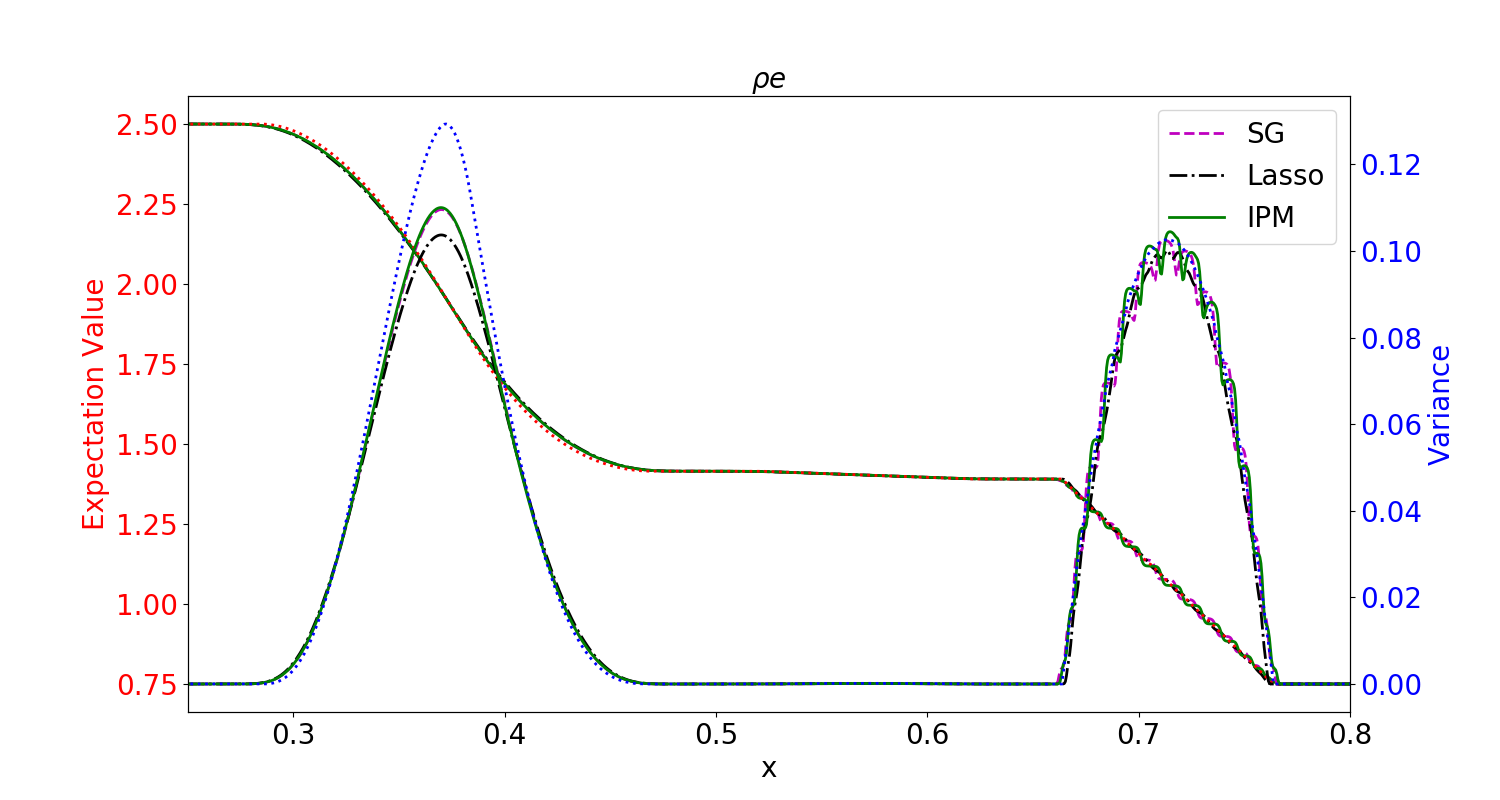}
  \caption{Expected value and variance of internal energy.}
  \label{fig:ExpVarEnergy}
\end{figure}

\begin{figure}[h!]
\centering
  \includegraphics[width=0.8\linewidth]{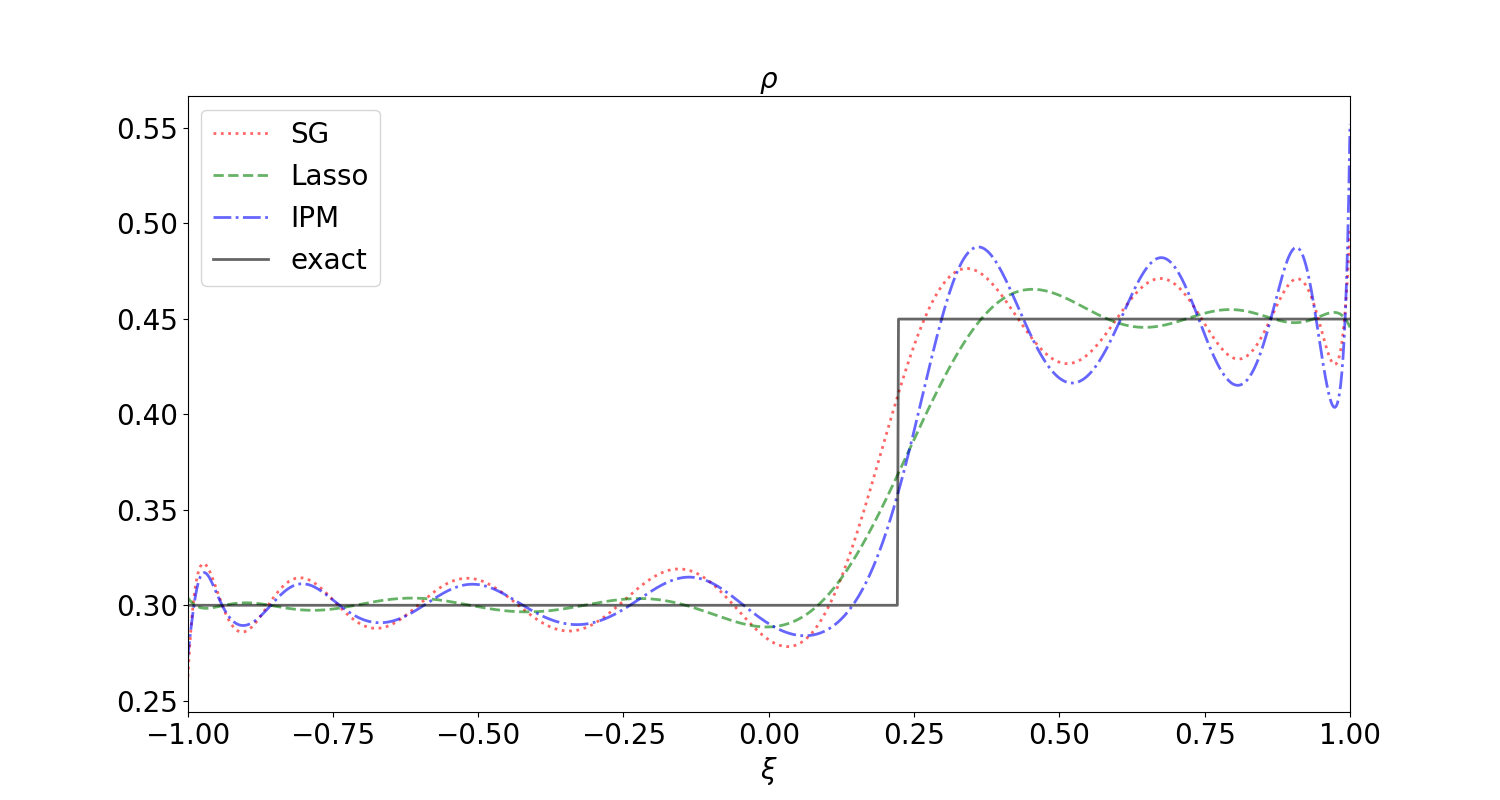}
  \caption{Density at fixed spatial position $x^* = 0.37$.}
  \label{fig:IC1fixedX}
\end{figure}
We start by looking at the expected value of $\rho$ in Figure \ref{fig:ExpVarRho}. The exact solution (which is the dotted red line) shows the expected value for the rarefaction wave in the left, the contact discontinuity in the middle and the shock on the right side of the spatial domain. A non-zero variance is observed at these solution regions. The exact variance is plotted by the blue dotted line. All three methods show a poor approximation of the variance, which is caused by the numerical diffusion of the finite volume scheme. The expected value is (except for the shock) nicely approximated by all three methods. A zoomed view of the shock can be found in Figure \ref{fig:ExpVarRhoShock}. While SG and IPM show a step-like approximation, the Lasso yields a satisfactory approximation. The same holds for the variance of the shock, where SG and IPM show oscillations. Looking at the variance of the rarefaction wave, one notices that SG and IPM reach the value of the variance more closely. The same behavior can be found for the momentum in Figure \ref{fig:ExpVarMoment} and the internal energy in Figure \ref{fig:ExpVarEnergy}. Looking at the approximation of the density shock for a fixed spatial position $x^* = 0.37$, one finds that the Lasso approximation, which is of polynomial order $N-1$ smears out the discontinuity, however damps oscillatory over- and undershoots while capturing the shock position nicely. The SG and IPM solutions show increased oscillations. Note that IPM especially oscillates at the right state, which leads to high density values. 
One needs to point out that in the chosen setting, the IPM yields no clear advantage compared to SG, since the choice of entropies is limited and we cannot prescribe upper and at the same time lower bounds as done for scalar problems. At the same time, the SG yields a similar result in a small portion of the computing time of IPM. However, for small densities both, the SG and Lasso will crash. By combining the Lasso method with a hyperbolicity preserving limiter \cite{schlachter2017hyperbolicity}, we are able to maintain hyperbolicity.

\subsection{Lightning strike with Obstacles}
\label{sec:Lightning}
Now, we study the stochastic Euler equations in 2D, which are
\begin{align*}
\partial_t
\begin{pmatrix}
\rho \\ \rho u \\ \rho v \\ \rho e
\end{pmatrix}
+\partial_x
\begin{pmatrix}
\rho u \\ \rho u^2 +p \\ \rho u v \\ \rho u (e+p)
\end{pmatrix}
+\partial_y
\begin{pmatrix}
\rho u \\ \rho u v +p \\ \rho v^2 \\ \rho v (e+p)
\end{pmatrix}
=\bm{0},
\end{align*}
with the initial conditions
\begin{align}\label{eq:ICEulerShock}
\rho_{\text{IC}} &= \begin{cases} \rho_L &\mbox{if } \Vert x \Vert < x_0 + \sigma \xi \\
\rho_R & \mbox{else } \end{cases} \\
(\rho u)_{\text{IC}} &= 0 \\
(\rho e)_{\text{IC}} &= \begin{cases} \rho_L e_L &\Vert x \Vert < x_0 + \sigma \xi \\
\rho_R e_R & \mbox{else } \end{cases}
\end{align}
The spatial domain is given by $[a,b]\times [a,b]$ and includes four square obstacles centered at positions $\bm{x}_{1,2,3,4}$ with length $l_{1,2,3,4}$. At the obstacles' boundaries, we use the Euler slip boundary condition. We use the following parameter values:
\begin{center}
    \begin{tabular}{ | l | p{7cm} |}
    \hline
    $[a,b]=[-0.3,0.3]$ & range of spatial domain \\
    $N_x=700,N_y = 700$ & number of spatial cells in each dimension \\
    $t_\mathrm{end}=0.14$ & end time \\
    $x_0 = 0.05, \sigma = 0.05$ & interface position parameters\\
    $\rho_L,p_L = 1.0, \rho_R = 0.8, p_R = 0.3$ & initial states\\
    $N = 8$ & polynomial degree\\
    $\bm{x}_{1} = (0,0.15)^T,\bm{x}_{2} = (0.1,0)^T,$ & obstacle positions\\ 
    $\bm{x}_{3} = (-0.1,0.1)^T,\bm{x}_{4} = (-0.1,0)^T$ & \\
    $l_1 = 0.06, l_2 = 0.04, l_3 = 0.02, l_4 = 0.01$ & obstacle length \\
    \hline
    \end{tabular}
\end{center}
Due to the fact that the spatial discretization consists of $N_x\cdot N_y$ cells and we have one additional equation (namely for the y-momentum), we do no longer study the results for IPM due to the much higher computational cost, and focus on comparing SG and the Lasso method. A reference solution has been computed using collocation with a $40$ point Gauss-Lobatto quadrature set. The different columns of Figure~\ref{fig:Obstacles2D} depict different methods to compute the solution, whereas the rows show the expectation value on the left and the variance on the right for the density. For both, expectation value and variance, Stochastic Galerkin yields oscillatory solutions. This is most obvious in the outer shock wave, but also reflected shocks suffer from oscillations as well. The Lasso filter yields results which agree nicely with the reference solution in both expected value and variance.

\newpage

\newgeometry{top=1.5cm, left=1.0cm}
\begin{figure}[h!]
\centering
	\begin{subfigure}{0.5\linewidth}
		\centering
		\includegraphics[scale=0.38]{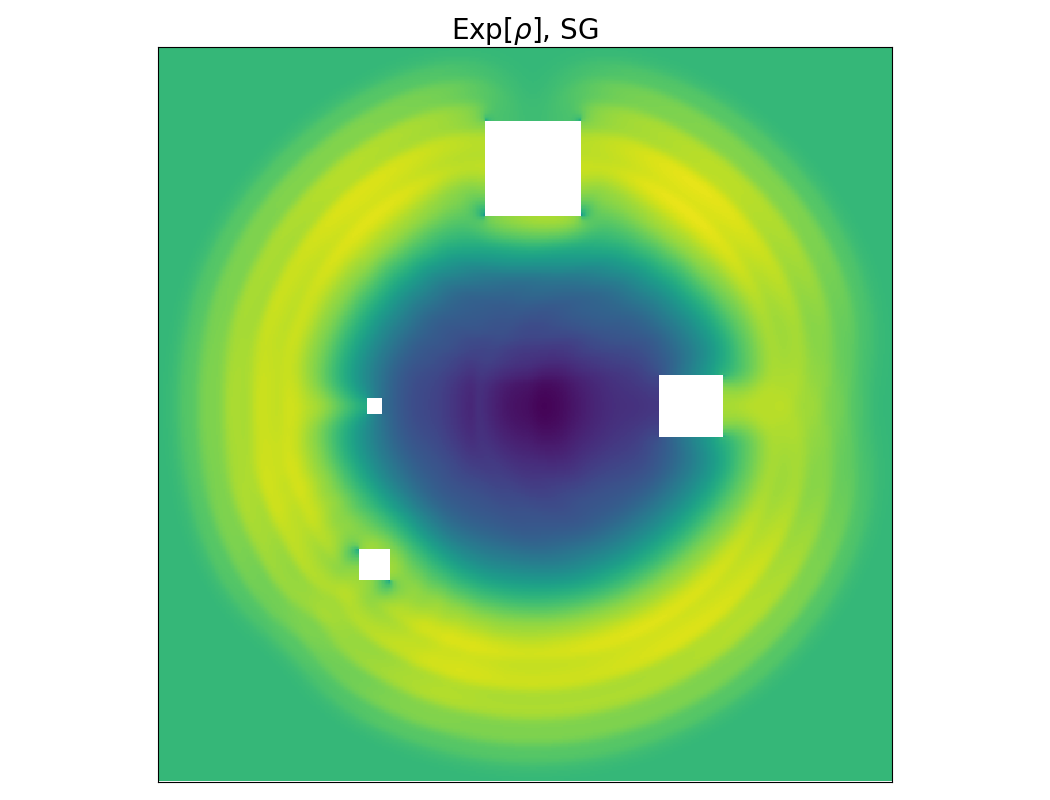}
		
		\label{fig:sub1}
	\end{subfigure}%
	\begin{subfigure}{0.5\linewidth}
		\centering
		\includegraphics[scale=0.38]{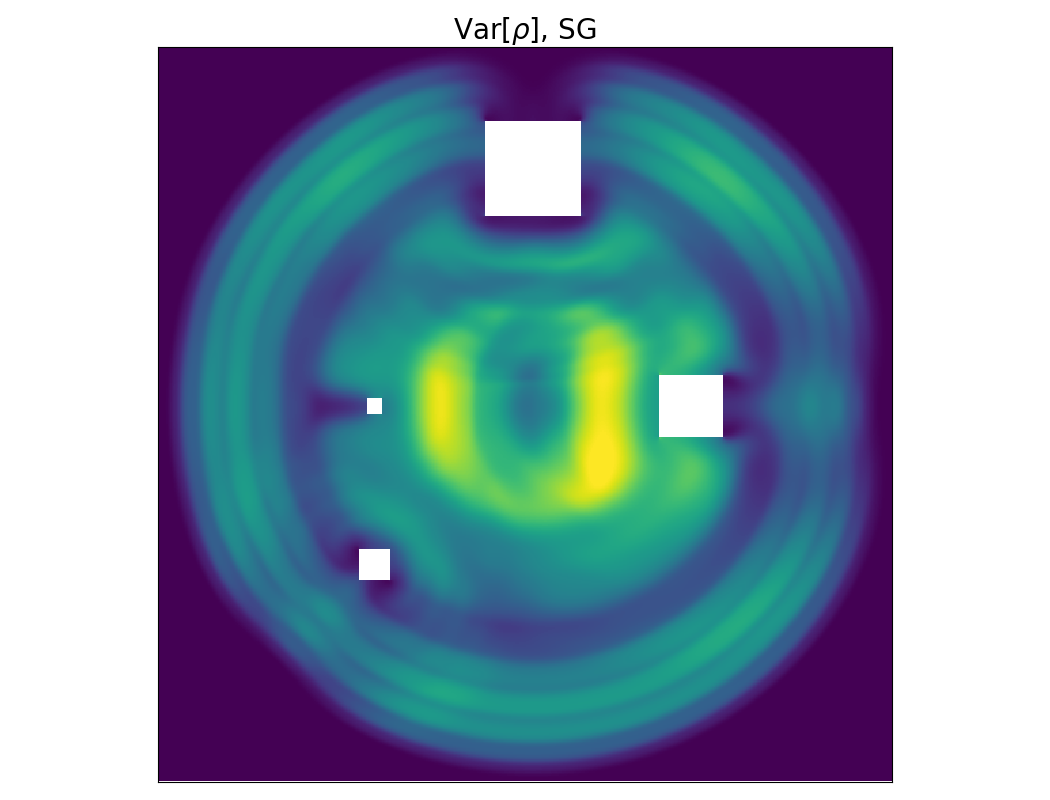}
		
		\label{fig:sub2}
	\end{subfigure}
	\begin{subfigure}{0.5\linewidth}
		\centering
		\includegraphics[scale=0.38]{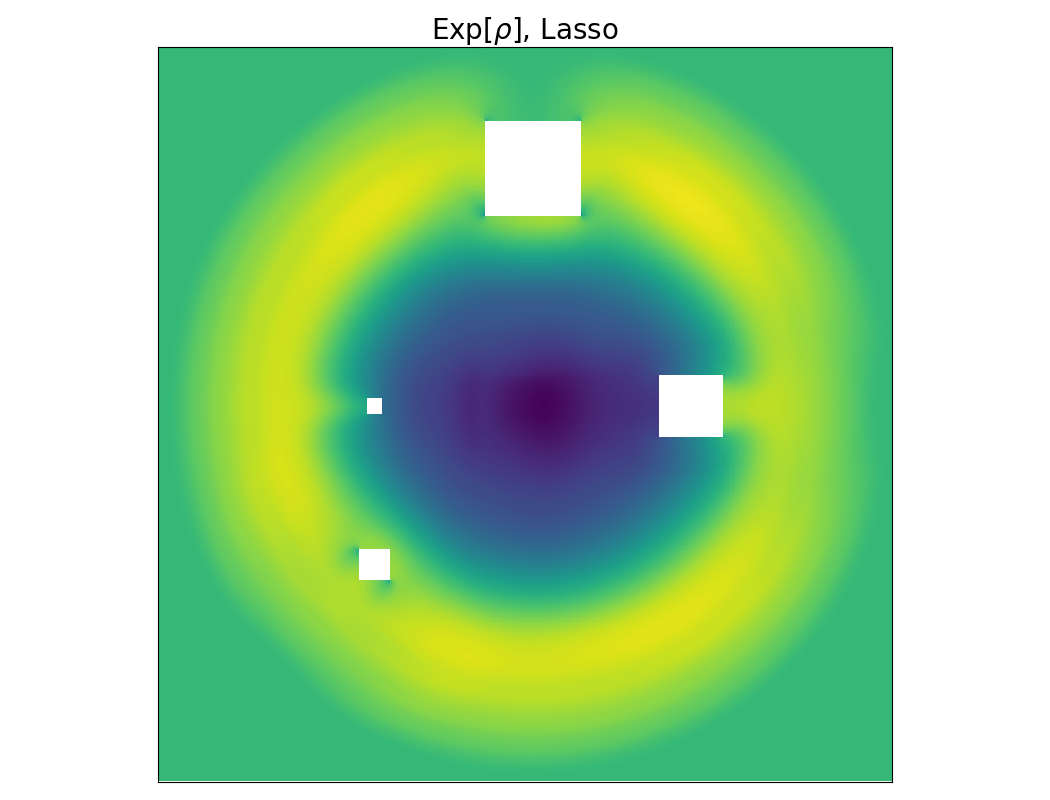}
		
		\label{fig:sub1}
	\end{subfigure}%
	\begin{subfigure}{0.5\linewidth}
		\centering
		\includegraphics[scale=0.38]{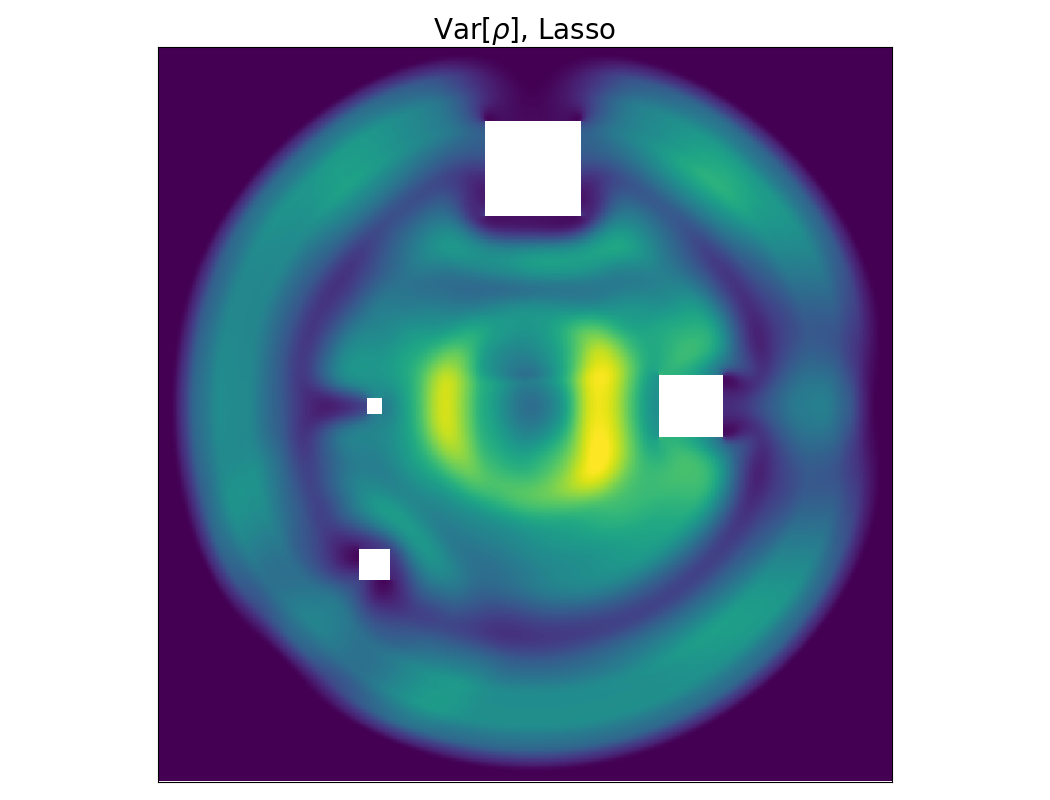}
		
		\label{fig:sub2}
	\end{subfigure}
		\begin{subfigure}{0.5\linewidth}
		\centering
		\includegraphics[scale=0.38]{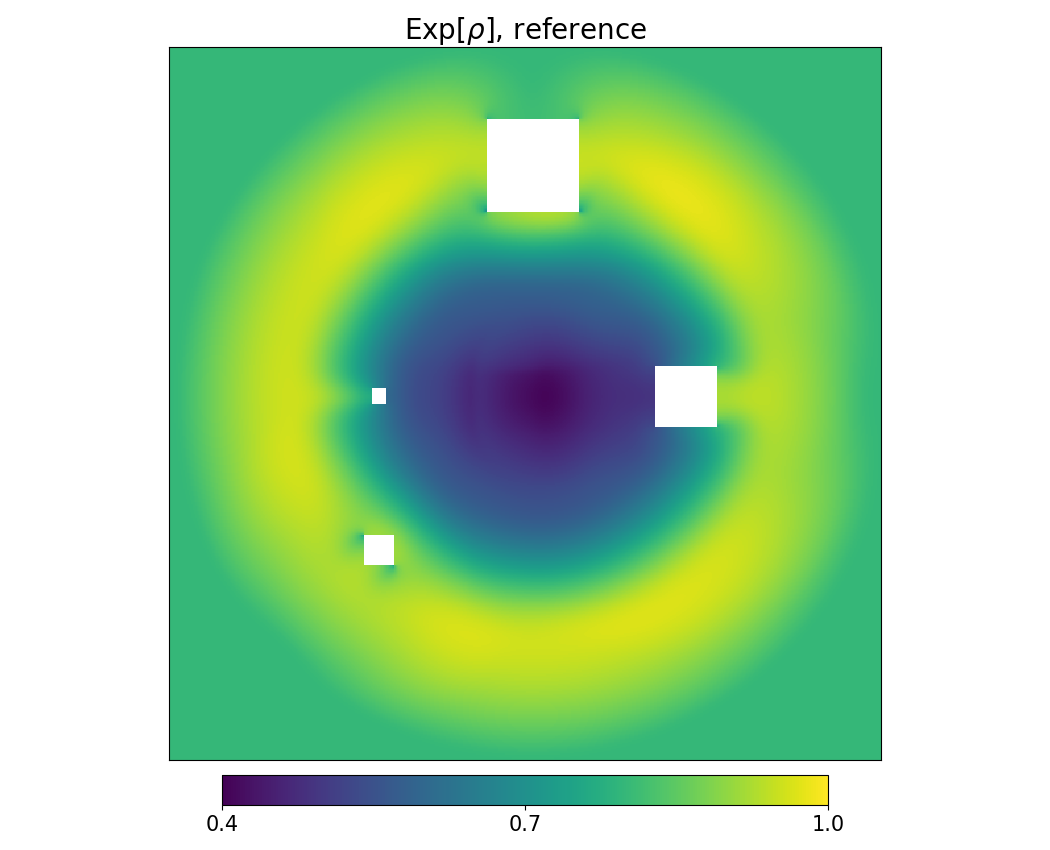}
		
		\label{fig:sub1}
	\end{subfigure}%
	\begin{subfigure}{0.5\linewidth}
		\centering
		\includegraphics[scale=0.38]{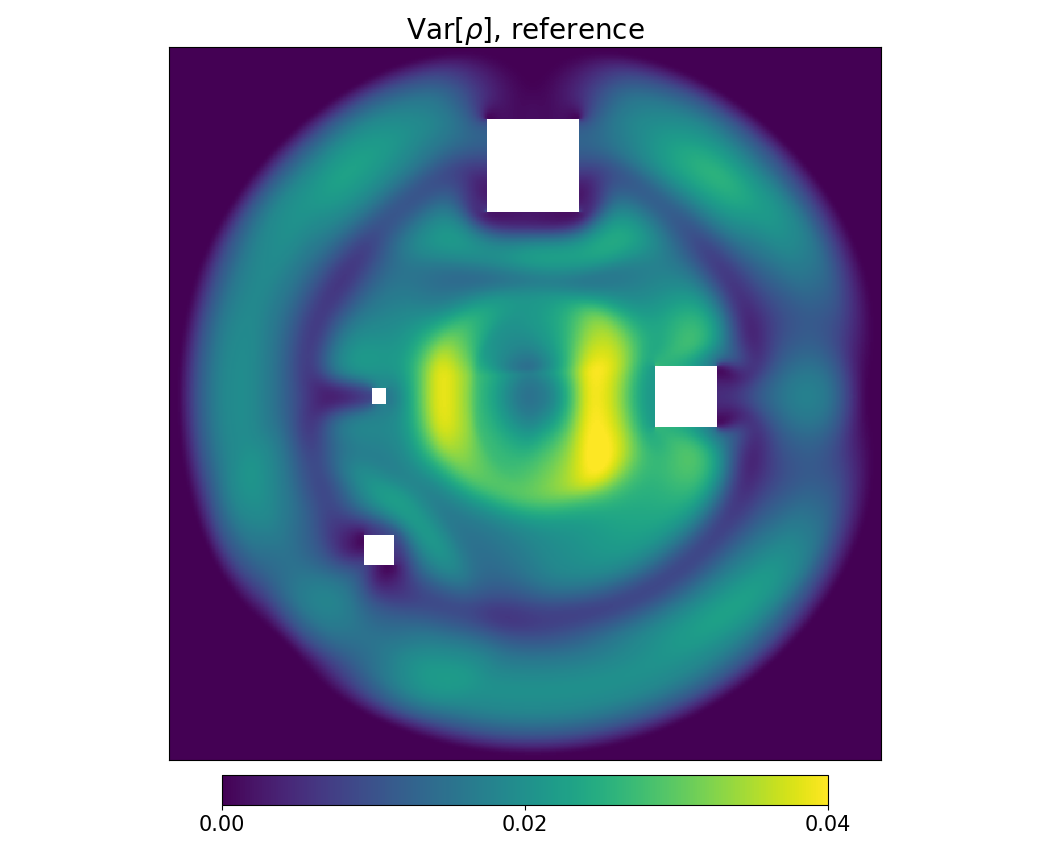}
		
		\label{fig:sub2}
	\end{subfigure}
	\caption{Expected value and variance with different methods plotted over the spatial domain.}
	\label{fig:Obstacles2D}
\end{figure}
\newgeometry{top=2.5cm, bottom=2.5cm}

\subsection{Shock in a duct}
The following test case is dedicated to comparing results obtained with different methods, including the L$^2$ filter. As before, we solve the two-dimensional Euler equations. As geometry, a duct is chosen. The initial condition is \eqref{eq:ICEulerShock} with $x_0=0.5$, i.e. the gas is initially in a shock state with high density and pressure on the left hand side.
Parameters, which differ from the settings in subsection~\ref{sec:Lightning} are
\begin{center}
    \begin{tabular}{ | l | p{7cm} |}
    \hline
    $[a,b]=[0.0,1.0]$ & range of spatial domain \\
    $N_x=400,N_y = 400$ & number of spatial cells in each dimension \\
    $t_\mathrm{end}=0.35$ & end time \\
    $x_0 = 0.5, \sigma = 0.1$ & interface position parameters\\
    \hline
    \end{tabular}
\end{center}
To use the L$^2$ filter, we need to conduct a parameter study to obtain an adequate filter coefficient $\lambda$. Due to the test case's similarity to the one-dimensional shock tube, we can perform a parameter study for a one-dimensional shock using the same numerical parameters as in two dimensions. Because of the heavily reduced numerical costs, we were able to find a suitable filter parameter of $\lambda = 3.0\cdot 10^{-6}$. The Lasso filter, as before, picks the filter parameter automatically. The results of the expectation value can be found in Figure~\ref{fig:ExpDuct2D} and the variance is depicted in Figure~\ref{fig:VarDuct2D}.
\begin{figure}[h!]
\centering
	\begin{subfigure}{0.5\linewidth}
		\centering
		\includegraphics[scale=0.27]{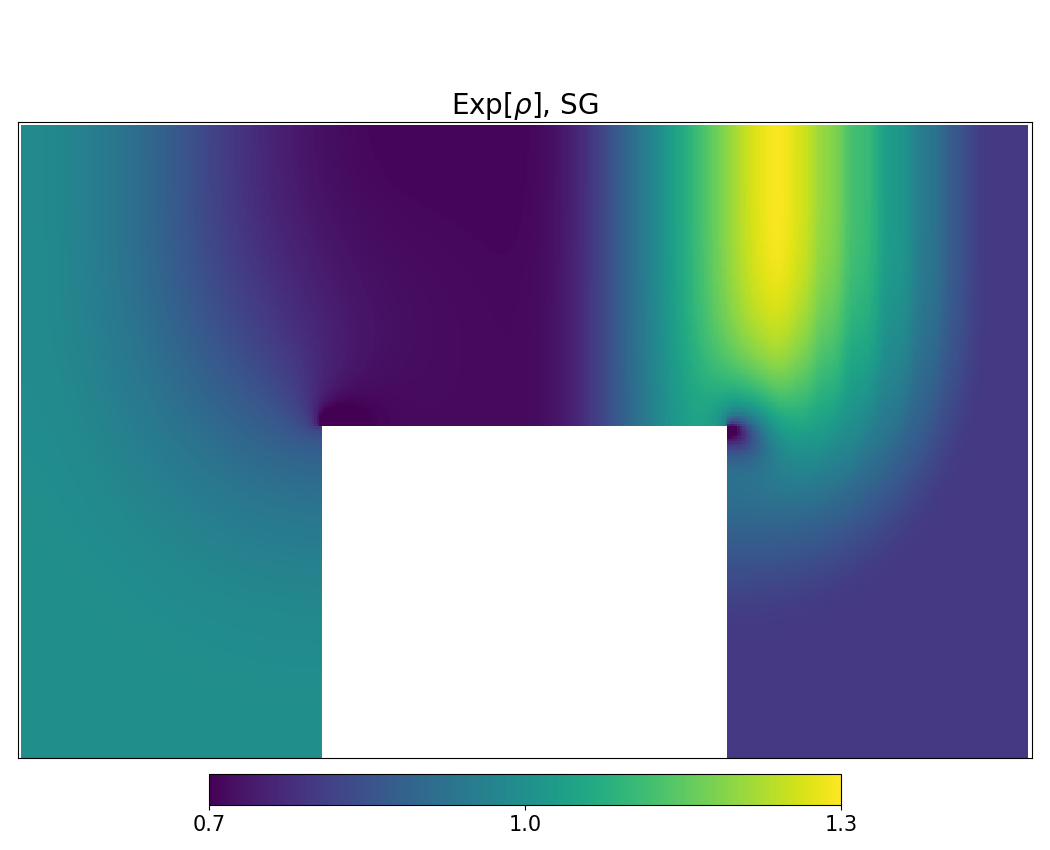}
		
		\label{fig:sub1}
	\end{subfigure}%
	\begin{subfigure}{0.5\linewidth}
		\centering
		\includegraphics[scale=0.27]{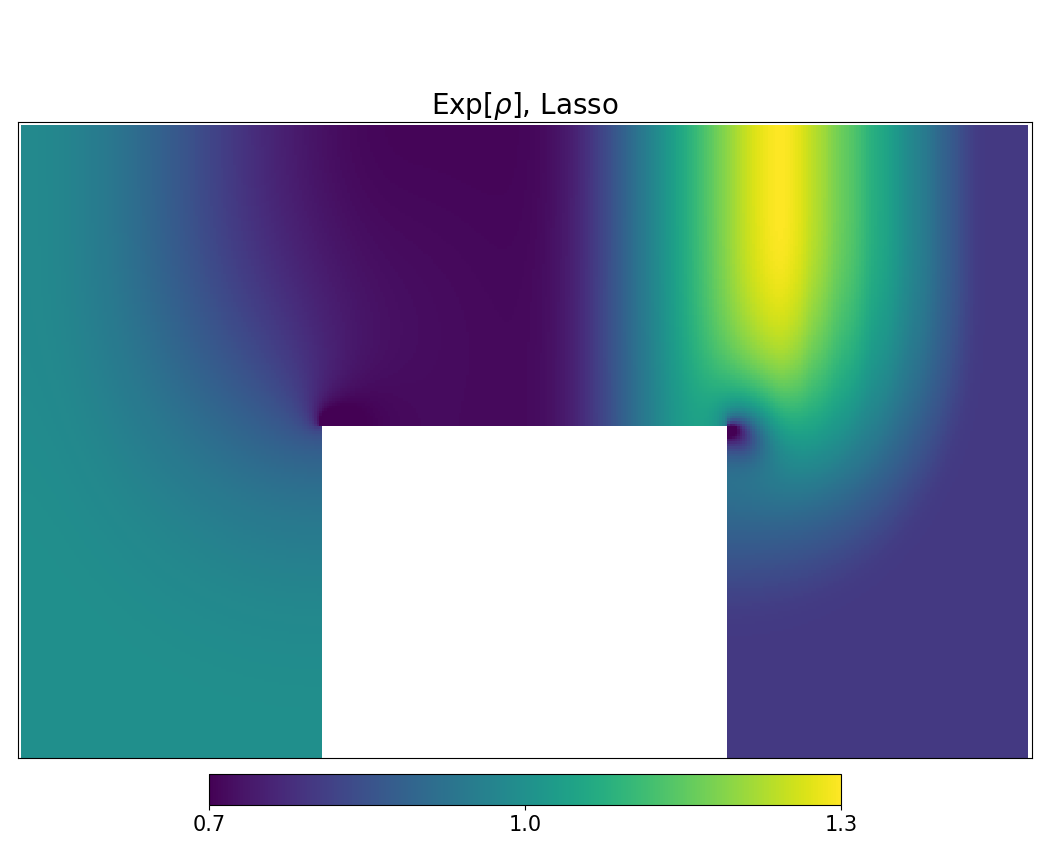}
		
		\label{fig:sub2}
	\end{subfigure}
	\begin{subfigure}{0.5\linewidth}
		\centering
		\includegraphics[scale=0.27]{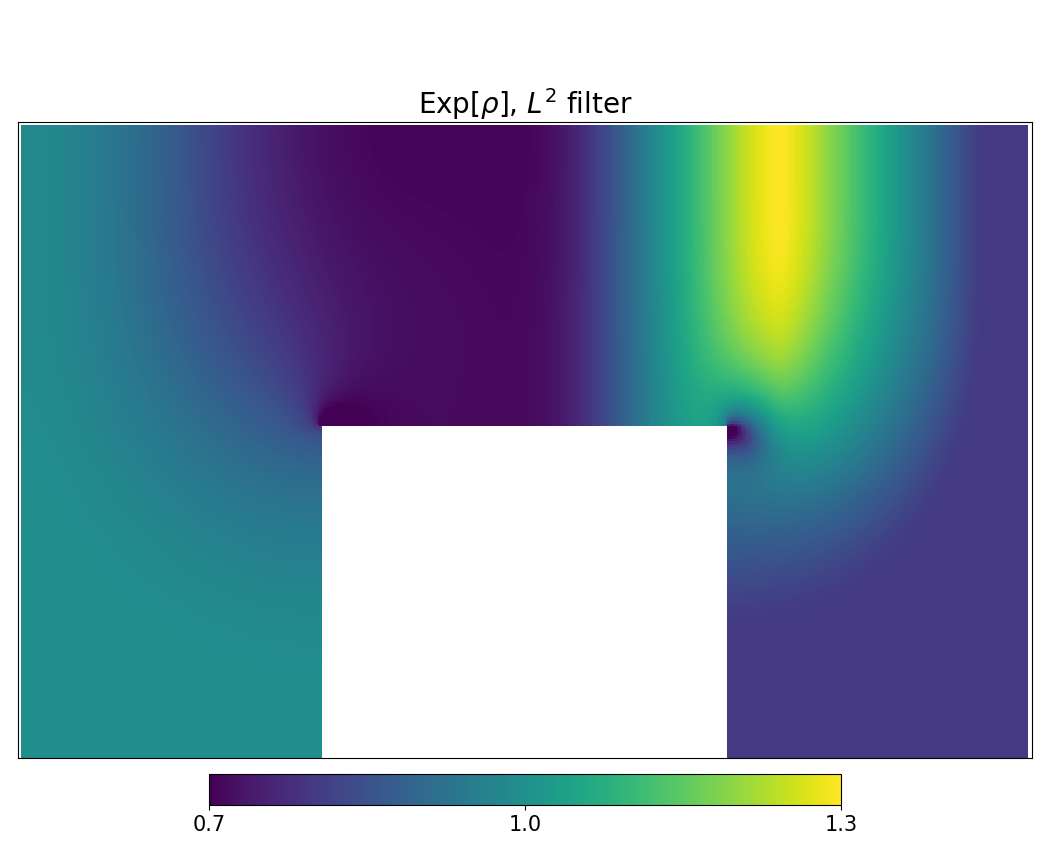}
		
		\label{fig:sub1}
	\end{subfigure}%
	\begin{subfigure}{0.5\linewidth}
		\centering
		\includegraphics[scale=0.27]{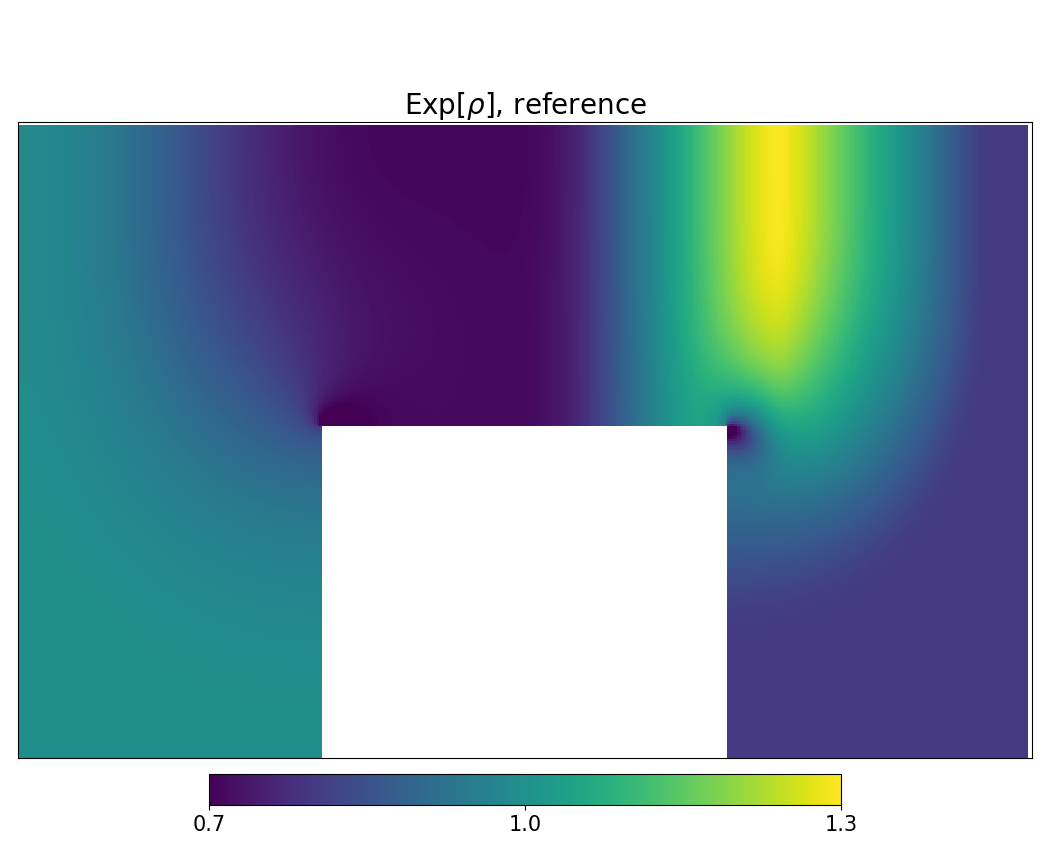}
		
		\label{fig:sub2}
	\end{subfigure}
	\caption{Expected value with different methods plotted over the spatial domain $[0,1] \times [0.3725,1]$.}
	\label{fig:ExpDuct2D}
\end{figure}
\begin{figure}[h!]
\centering
	\begin{subfigure}{0.5\linewidth}
		\centering
		\includegraphics[scale=0.27]{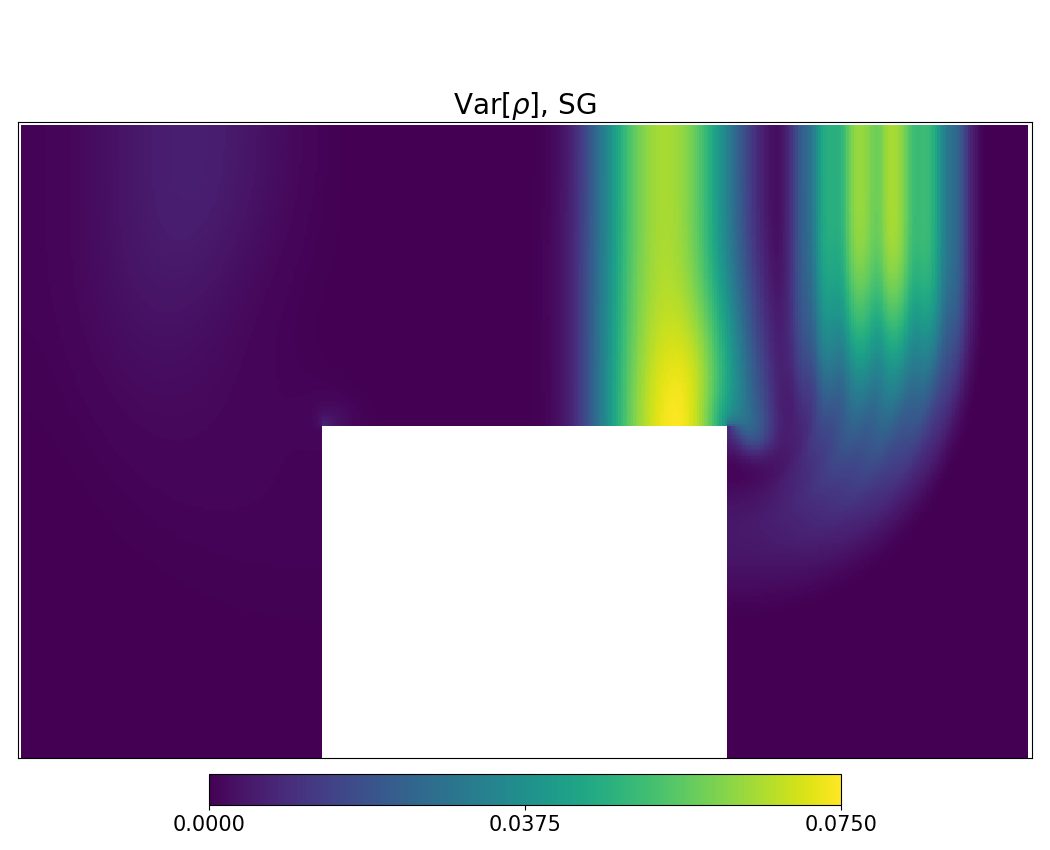}
		
		\label{fig:sub1}
	\end{subfigure}%
	\begin{subfigure}{0.5\linewidth}
		\centering
		\includegraphics[scale=0.27]{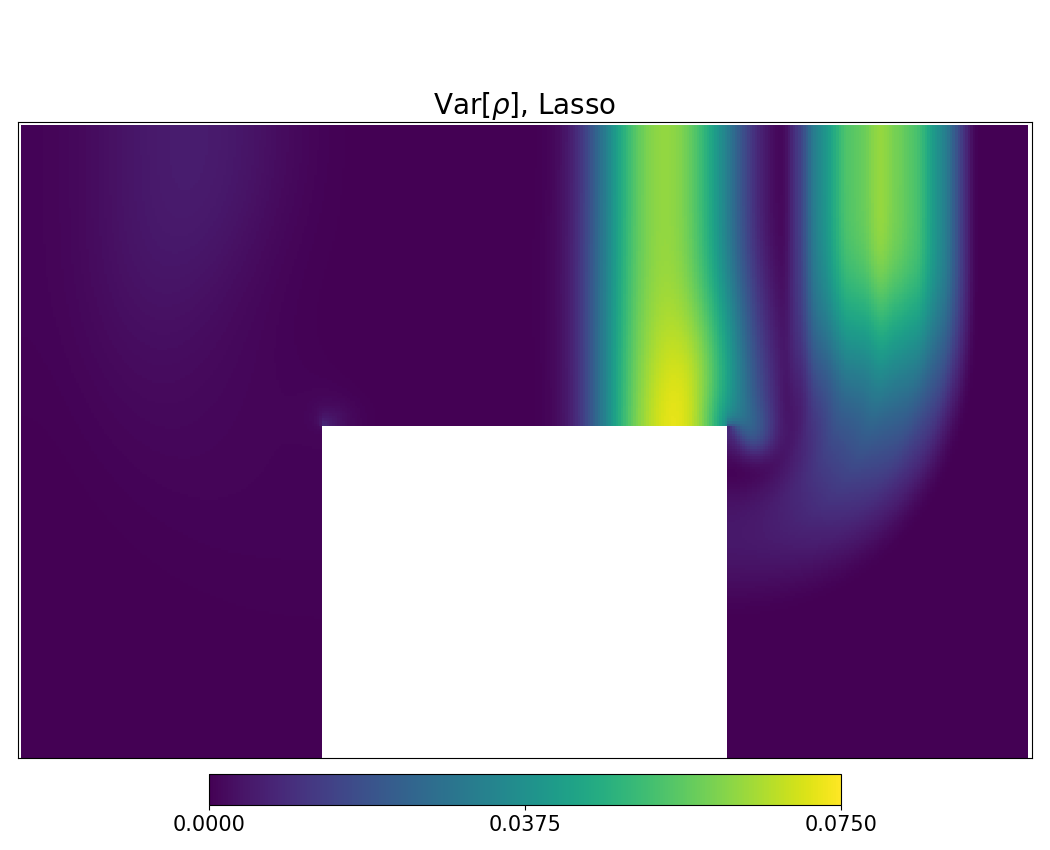}
		
		\label{fig:sub2}
	\end{subfigure}
	\begin{subfigure}{0.5\linewidth}
		\centering
		\includegraphics[scale=0.27]{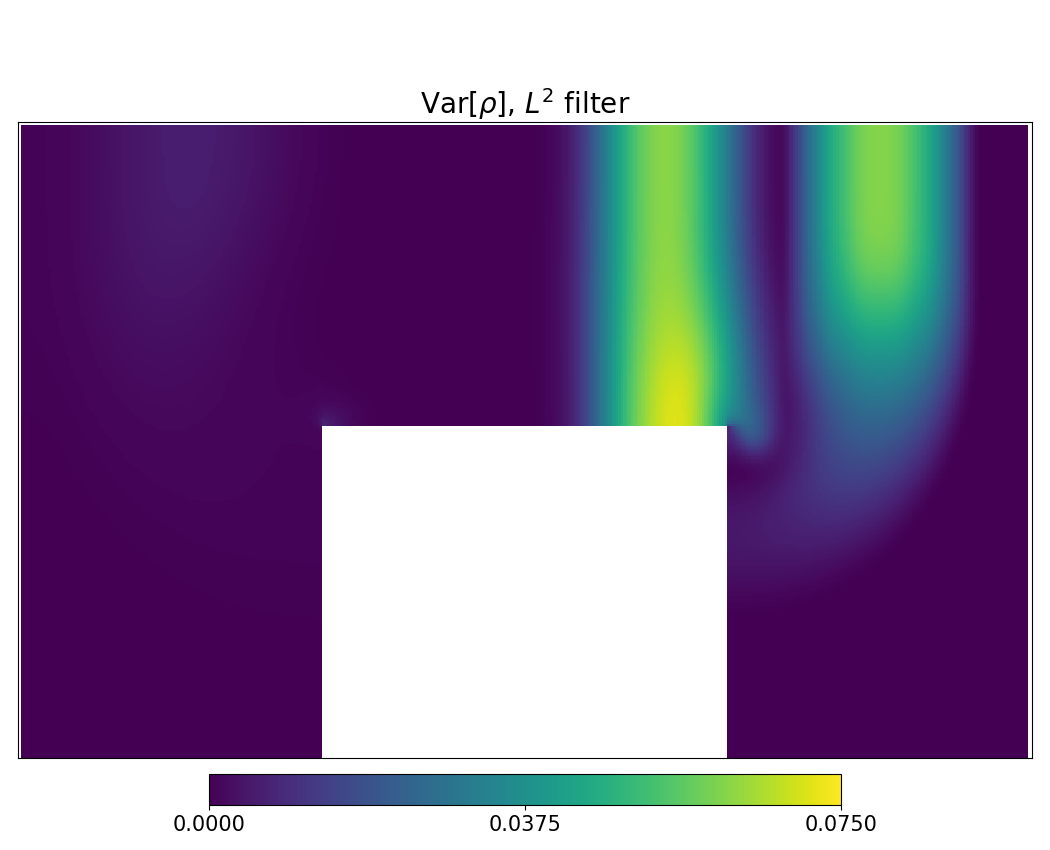}
		
		\label{fig:sub1}
	\end{subfigure}%
	\begin{subfigure}{0.5\linewidth}
		\centering
		\includegraphics[scale=0.27]{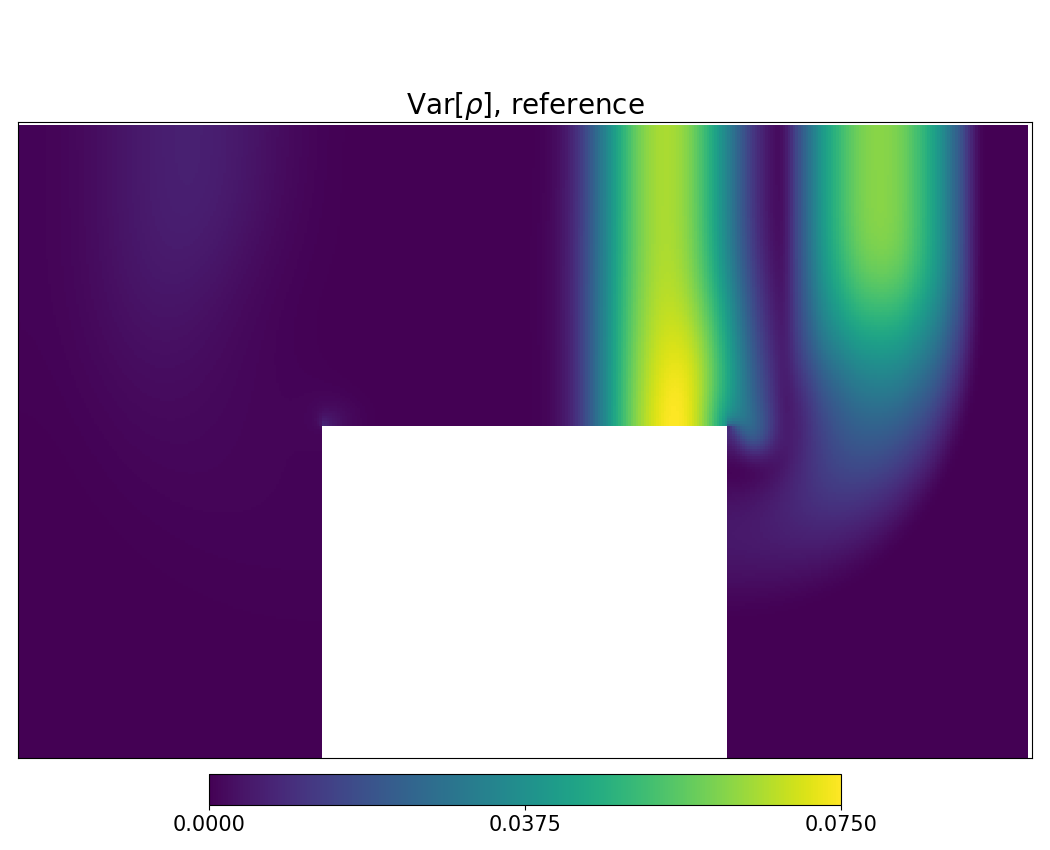}
		
		\label{fig:sub2}
	\end{subfigure}
	\caption{Variance with different methods plotted over the spatial domain $[0,1] \times [0.3725,1]$.}
	\label{fig:VarDuct2D}
\end{figure}

Again, the SG solution of both expected value and variance shows non-physical oscillations at the front shock. Both filters are able to mitigate these oscillations, however, the L$^2$ filter yields more accurate results. Note, however that the chosen filter strength was determined by a parameter study, which in general is prohibitively expensive if one cannot use simpler, one-dimensional problems to model filtering effect of two-dimensional problems.

\section{Conclusion and Outlook}
\label{sec:ConclusionOutlook}
In this work, we have investigated the effects of filters when being applied to the SG method. Due to the challenging task of determining a suitable filter strength for a given problem, we derived the Lasso filter. This filter uses a Lasso penalizing term in the cost function to enforce sparsity of the filtered expansion coefficients. Taking advantage of this property, we automated the choice of the filter coefficients by choosing a filter strength which sets the last moment to zero. Consequently, the filter is turned down in smooth solution regimes (which have a small last moment) and amplified at shocks. We have applied the Lasso filter to several test cases, where we compare the results with IPM and standard SG. It turns out that the filter outperforms SG and can compete with IPM in scalar problems. The investigation of the Euler equations showed that both IPM and SG show a non-physical, step-like profile of the expectation value at shocks, whereas the Lasso filter satisfactorily approximates the correct linear behavior. However, the variance of the rarefaction wave is better resolved by SG and IPM. In the two-dimensional setting, we again observed oscillations of SG for both expectation value and variance. In contrast to that, the Lasso filter yields non-oscillatory results and shows good agreement with the reference solution. The same holds for the L$^2$ filter, where we chose the filter strength based on a one-dimensional parameter study.

Due to their easy integration into existing SG code as well as their nice approximation behavior, we consider filters to be a promising tool in uncertainty quantification. Various properties should be examined in future work: First of all, one should investigate the effects of different filters and find ways to automate the choice of the filtering coefficients. This can be done by forcing the last moments to lie below a certain threshold or by investigating the regularity of the corresponding equation solved by the filter analytically. Furthermore, we have not shown results of the filter combined with the hyperbolicity limiter introduced in \cite{schlachter2017hyperbolicity}. The idea is to use this limiter on the filtered moments in order to preserve hyperbolicity of the moment system. One needs to study the effects caused by the combination of these two techniques on the solution. Furthermore, the effects of different filters when investigating problems with multiple uncertainties needs to be investigated.

\section*{Acknowledgments}
Funding: Jonas Kusch and Martin Frank were supported by the German Research Foundation (DFG) under grant FR 2841/6-1. Ryan McClarren would like to thank the Steinbuch Centre for Computing for hosting him during portions of this work.

\bibliographystyle{siamplain}

\bibliography{Lasso}
\end{document}